\documentclass[11pt]{amsart}

\numberwithin{equation}{section}

\usepackage{amssymb,amscd} \usepackage{enumerate, xspace, graphicx}
\usepackage{changebar}
\usepackage{hyperref}
\usepackage{cite}

\usepackage{color}

\newtheorem{theo}{Theorem}
\newtheorem{lemma}{Lemma} \newtheorem{coro}{Corollary}

\newtheorem{prop}{Proposition}

\theoremstyle{definition}

\newtheorem{exmp}{Example}

\theoremstyle{remark}
\newtheorem{remark}{Remark}

\newcommand{\eqbox}[2]{\text{\begin{minipage}[#1]{0.88\linewidth}#2\end{minipage}}}

\newcommand\Dc{{\mathcal D}}

\newcommand\Lc{{\mathcal L}}
\newcommand\Oc{{\mathcal O}}

\newcommand\Fb{{\mathbf F}}
\newcommand\Tb{{\mathbf T}}

\newcommand\N{{\mathbb N}} 
\newcommand\R{{\mathbb R}} \newcommand\C{{\mathbb C}}

\newcommand{\x}{\mathbf{x}}

\newcommand{\conv}{\operatorname{conv}}
\newcommand{\diag}{\operatorname{diag}}
\newcommand{\supp}{\operatorname{supp}}
\newcommand{\id}{\operatorname{id}}
\newcommand{\Var}{\operatorname{{Var}}}
\newcommand{\interior}[1]{\operatorname{int(#1)}}
\newcommand{\assign}{:=}
\newcommand\transpose{{\!\scriptscriptstyle\mathsf T}}
\renewcommand{\leq}{\leqslant}
\renewcommand{\geq}{\geqslant}

\renewcommand{\d}{\mathsf{d}_W}
\newcommand{\BV}{{\rm BV}}
\begin{document}

\title[Globally coupled maps with bistable thermodynamic limit]{Stochastically stable globally coupled maps with bistable thermodynamic limit}
\author{Jean-Baptiste Bardet, Gerhard Keller and Roland Zweim\"uller}
\noindent \address{J.-B. Bardet: IRMAR, Universit\'{e} Rennes 1, Campus de
  Beaulieu, 35042 Rennes Cedex, France and LMRS, Universit\'e de Rouen, Avenue de l'Universit\'e, BP.12, Technop\^ole du Madrillet,
76801 Saint-\'{E}tienne-du-Rouvray, France;
\newline
G. Keller: Department Mathematik,
  Universit\"{a}t Erlangen-N\"{u}rnberg, Bismarckstr. 1 1/2, 91054 Erlangen,
  Germany;
\newline
 R. Zweim\"uller: Fakult\"at f\"ur Mathematik,
Universit\"at Wien,
Nordbergstrasse 15,
1090 Wien, Austria.}
\noindent \email{{\tt jean-baptiste.bardet@univ-rouen.fr,
    keller@mi.uni-erlangen.de, roland.zweimueller@univie.ac.at }}
\date{\today}
\thanks{This cooperation was supported
  by the DFG grant Ke-514/7-1 (Germany).  J.-B.B. was also partially supported
  by CNRS (France). The authors acknowledge the hospitality of the ESI
  (Austria) where part of this research was done.  G.K. thanks Carlangelo
  Liverani for a discussion that helped to shape the ideas in
  Section~\ref{subsec:differentiability}.  }

\begin{abstract}
  We study systems of globally coupled interval maps, where the identical
  individual maps have two expanding, fractional linear, onto branches, and
  where the coupling is introduced via a parameter - common to all individual
  maps - that depends in an analytic way on the mean field of the system. We
  show: 1) For the range of coupling parameters we consider, finite-size
  coupled systems always have a unique invariant probability density which is
  strictly positive and analytic, and all finite-size systems exhibit exponential
  decay of correlations. 2) For the same range of parameters, the self-consistent
  Perron-Frobenius operator which captures essential aspects of the
  corresponding infinite-size system (arising as the limit of the above
  when the system size tends to infinity), undergoes a supercritical
  pitchfork bifurcation from a unique stable equilibrium to the coexistence of
  two stable and one unstable equilibrium.
\end{abstract}

\maketitle

\section{Introduction}
\label{sec:intro}

\emph{Globally coupled maps} are collections of individual discrete-time
dynamical systems (their \emph{units}) which act independently on their
respective phase spaces, except for the influence (the \emph{coupling}) of a
common parameter that is updated, at each time step, as a function of the
\emph{mean field} of the whole system. Systems of this type have received some
attention through the work of Kaneko \cite{Kaneko1,Kaneko2} in the early
1990s, who studied systems of $N$ quadratic maps acting on coordinates
$x_1,\dots,x_N\in[0,1]$, and coupled by a parameter depending in a simple way
on $\bar x:=N^{-1}(x_1+\dots+x_N)$. His key observation, for huge \emph{system
  size} $N$, was the following: if $(\bar x^t)_{t=0,1,2,\dots}$ denotes the
time series of mean field values of the system started in a random
configuration $(x_1,\dots,x_N)$, then, for many parameters of the quadratic
map, and even for very small coupling strength, pairs $(\bar x^t,\bar
x^{t+1})$ of consecutive values of the field showed complicated functional
dependencies plus some noise of order $N^{-1/2}$, whereas for uncoupled
systems of the same size the $\bar x^t$, after a while, are constant up to
some noise of order $N^{-1/2}$. While the latter observation is not surprising
for independent units, the complicated dependencies for weakly coupled
systems, a phenomenon Kaneko termed \emph{violation of the law of large
  numbers}, called for closer investigation.

The rich bifurcation structure
of the family of individual quadratic maps may offer some
explanations, but since a mathematically rigorous investigation of even a
small number of coupled quadratic maps in the chaotic regime still is a
formidable task, there seem to be no serious attempts to tackle this problem.

A model which is mathematically much easier to treat is given by coupled tent
maps. Indeed, for tent maps with slope larger than $\sqrt2$ and moderate
coupling strength, a system of $N$ mean field coupled units has an ergodic
invariant probability density with exponentially decreasing correlations
\cite{Keller97}. This is true for all $N$ and for coupling strengths that can be
chosen to be the same for all $N$. Nevertheless, Ershov and Potapov
\cite{ErPo1} showed numerically that (albeit on a much smaller length scale
than in the case of coupled quadratic maps) also mean field coupled tent maps
exhibit a violation of the law of large numbers in the aforementioned sense. They also
provided a mathematical analysis which demonstrated that the discontinuities of the
invariant density of a tent map are at the heart of the problem. Their
analysis was not completely rigorous, however, as Chawanya and Morita
\cite{ChaMo} could show that there are indeed (exceptional) parameters of the
system for which there is no violation of the law of large numbers -
contrary to the predictions in \cite{ErPo1}. On the other hand, references
\cite{NaKo1,NaKo2} contain further simulation results on systems violating
the law of large numbers. (But at present, a mathematically rigorous treatment of
globally coupled tent maps that is capable of classifying and explaining the
diverse dynamical effects that have been observed does not seem to be in
sight either.) These studies were complemented by papers by J\"arvenp\"a\"a
\cite{Jarvenpaa} and Keller \cite{Keller00}, showing (among other things)
that globally coupled systems of smooth expanding circle maps do not display
violation of the law of large numbers at small coupling strength, because
their invariant densities are smooth.

Given this state of knowledge, the present paper investigates specific systems
of globally coupled piecewise fractional linear maps on the interval
$X:=[-\frac12,\frac12]$, where each individual map has a smooth invariant
density.  For small coupling strength, Theorem 4 in \cite{Keller00} extends
easily to this setup and proves the absence of a violation of the law of large
numbers. For larger coupling strength, however, we are going to show that this
phenomenon does occur in the following sense:
\begin{description}
\item[Bifurcation] The nonlinear \emph{self-consistent Perron-Frobenius operator (PFO)}
  $\widetilde P$ on $L_1(X,\lambda)$, which describes the dynamics of the
  system in its thermodynamic limit, undergoes a supercritical pitchfork
  bifurcation as the coupling strength increases. (Here and in the sequel
  $\lambda$ denotes Lebesgue measure.)
\item[Mixing] At the same time, all corresponding finite-size systems have unique
  absolutely continuous invariant probability measures \boldmath$\mu$\unboldmath$_N$ on
  their $N$-dimensional state space, and exhibit exponential decay of correlations
  under this measure.
\item[Stable behaviour] In the \emph{stable regime}, i.e. for fixed small coupling
  strength below the bifurcation point of
  the infinite-size system, the measures \boldmath$\mu$\unboldmath$_N$
  converge weakly, as the system size $N\to\infty$, to an infinite product measure
  $(u_0\cdot\lambda)^{\mathbb N}$,
  where $u_0$ is the unique fixed point of $\widetilde P$.
\item[Bistable behaviour] In the \emph{bistable regime}, i.e. for fixed coupling
  strength above the bifurcation point of the infinite-size system,
  all possible weak limits of the measures
  \boldmath$\mu$\unboldmath$_N$ are convex combinations of the three infinite
  product measures $(u_r\cdot\lambda)^{\mathbb N}$, $r\in\{-r_*,0,r_*\}$, where now
  $u_0$ is the unique unstable fixed point of $\widetilde P$ and $u_{\pm r_*}$
  are its two stable fixed points. (We conjecture that the measure
  $(u_0\cdot\lambda)^{\mathbb N}$ is not charged in the limit.)
\end{description}
  This scenario clearly bears some resemblance to the Curie-Weiss model from
  statistical mechanics and its dynamical variants.

  We also stress that a simple modification of our system leads to a variant
  where, instead of two stable fixed points, one stable two-cycle for
  $\widetilde P$ is created at the bifurcation point.  This may be viewed as
  the simplest possible scenario for a violation of the law of large numbers
  in Kaneko's original sense.\\

In the next section we describe our model in detail, and formulate the main
results. Section~\ref{sec:finite-sys-proofs} contains the proofs for
finite-size systems. In Section~\ref{sec:ifs} we start the investigation of the
infinite-size system via the self-consistent PFO $\widetilde P$. We observe that
this operator preserves a class of probability densities which can be
characterised as derivatives of Herglotz-Pick-Nevanlinna functions.
Integral representations of these functions reveal a hidden order structure,
which is respected by the operator $\widetilde P$,
and allows us to describe the pitchfork bifurcation. In
Section~\ref{sec:infinite-sys-proofs} this dynamical picture for $\widetilde
P$ is extended to arbitrary densities.
Finally, in Section \ref{sec:noisy}, we
discuss the situation when some noise is added to the dynamics.

\section{{Model and main results}}
\label{sec:modelandresults}

\subsection{The parametrised family of maps}
\label{subsec:family-def}

Throughout, all measures are understood to be Borel, and
we let $\mathsf{P}(B):=\{$probability measures on $B\}$.
Lebesgue measure will be denoted by $\lambda$.
{We introduce} a $1$-parameter family of piecewise
fractional-linear transformations $T_{r}$ on
$X:=[-\frac{1}{2},\frac{1}{2}]${, which will play the role of
  the local maps}. To facilitate manipulation of such maps, we
{use their standard matrix representation, letting}
\[
f_{M}(x):=\frac{ax+b}{cx+d}\text{ \quad for any real $2\times2$-matrix
}M=\left(
\begin{array}
[c]{cc}%
a & b\\
c & d
\end{array}
\right)  \text{,}%
\]
so that $f_{M}^{\prime}(x)=(ad-bc)/(cx+d)^{2}$ and $f_{M}\circ f_{N}=f_{MN}$.
Specifically, we consider the function $f_{M_{r}}$, depending on a parameter
$r\in(-2,2)$, given by the coefficient matrix
\[
M_{r}:=\left(
\begin{array}
[c]{cc}%
r+4 & r+1\\
2r & 2
\end{array}
\right)  \text{.}%
\]
One readily checks that $f_{M_{r}}(-\frac{1}{2})=-\frac{1}{2}$, $f_{M_{r}%
}(\frac{1}{2})=\frac{3}{2}$, $f_{M_{r}}(\alpha_{r})=\frac{1}{2}$ for
$\alpha_{r}:=-r/4$, and that (the infimum being attained on $\partial X$)
\[
f_{M_{r}}^{\prime}(x)=\frac{4-r^{2}}{2\left(  rx+1\right)  ^{2}}\geq
2\,\frac{2-\left\vert r\right\vert }{2+\left\vert r\right\vert }=\inf
_{X}f_{M_{r}}^{\prime}>0\text{ \quad for }x\in X\text{.}%
\]
The latter shows that $f_{M_{r}}$ is uniformly expanding if and only if $\left\vert
r\right\vert <\frac{2}{3}$, and we define our \emph{single-site maps}
$T_{r}:X\rightarrow X$ with parameter $r\in(-2/3,2/3)$ by letting
\[
T_{r}(x):=f_{M_{r}}(x)\text{ mod }\left(  \mathbb{Z}+\frac{1}{2}\right)
=\left\{
\begin{array}
[c]{ll}%
f_{M_{r}}(x) & \text{on }[-\frac{1}{2},\alpha_{r})\text{,}\\
f_{M_{r}}(x)-1=f_{N_{r}}(x) & \text{on }(\alpha_{r},\frac{1}{2}]\text{,}%
\end{array}
\right.
\]
where
\[
N_{r}:=
\left(\begin{array}{cc}
  1&-1\\0&1
\end{array}\right)
M_r\ .
\]
We thus obtain a family $(T_{r})_{r\in(-2/3,2/3)}$ of uniformly
expanding, piecewise invertible maps $T_{r}:X\rightarrow X$, each
having two increasing covering branches. Note also that this family is
symmetric in that
\begin{equation} -T_{r}(-x)=T_{-r}(x)\text{ \quad for }r\in
 {\textstyle \left(
-\frac{2}{3},\frac{2} {3}\right)} \text{ and }x\in
X\text{.}\label{Eq_Symmetry}%
\end{equation}
According to well-known folklore results, each map
$T_{r}$, $r\in(-2/3,2/3)$, has a unique invariant probability density
$u_{r}\in\mathcal{D}:=\{u\in
L_{1}(X,\lambda):u\geq0,\int_{X}u\,d\lambda=1\}$, and $T_{r}$ is exact
(hence ergodic) w.r.t. the corresponding invariant measure. Due to
(\ref{Eq_Symmetry}), we have $u_{-r}(x)=u_{r}(-x)$ mod $\lambda$. We
denote the \emph{Perron-Frobenius operator} (\emph{PFO}),\
w.r.t. Lebesgue measure $\lambda$, of a map $T$ by $P_{T}$,
abbreviating $P_{r}:=P_{T_{r}}$. In our construction below we will
exploit the fact that 2-to-1 fractional linear maps like $T_{r}$ in
fact enable a fairly explicit analysis of their PFOs on a suitable
class of densities.  {In
particular, the $u_r$ are known explicitly:}
\begin{remark} Let $\gamma_{{r}}:=\frac{r}{1+r}$,
${\delta_r}:=\frac{r}{1-r}$. Then
  \begin{equation}
    \label{eq:inv-dens-1} \tilde u_r(x) :=
    \int_{\gamma_{{r}}}^{\delta_{{r}}}\frac1{(1-xy)^2}\,dy
    = \frac{2r^2}{(rx-(1-r))(rx-(1+r))}
  \end{equation}
  is an integrable invariant density for $T_r$, see \cite{Sch81}. Its
  normalised version
  \begin{equation}
    \label{eq:inv-dens-2}
u_r(x):=\left(\log\frac{r^2-4}{9r^2-4}\right)^{-1}\cdot \tilde u_r(x)
  \end{equation} is the unique $T_r$-invariant probability density.
\end{remark}

{The key point in the choice of this family of maps is
that for $r<0$,
$T_r$ is steeper in the positive part of $X$
{than} in its negative part,
hence typical orbits spend more time on the negative part, which is
confirmed by the invariant density (see Figure \ref{fig:Trandhr}).
If $r>0$, then $T_r$ favours the positive part.}

\begin{figure}
  \begin{minipage}{0.325\linewidth} \centering
    \includegraphics[width=0.95\linewidth,height=0.95\linewidth]{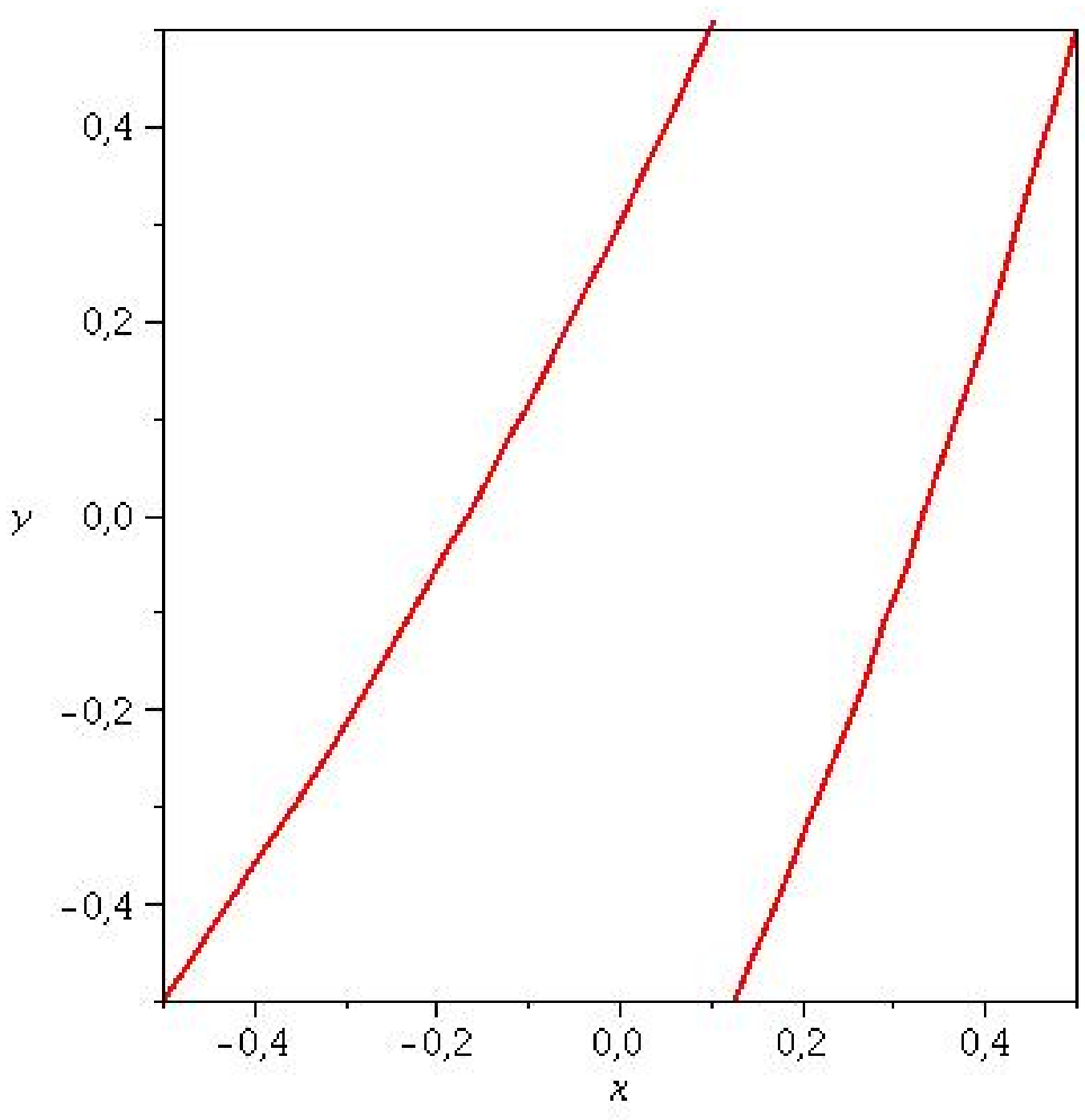}
  \end{minipage}
\hspace{0.1\linewidth}
  \begin{minipage}{0.325\linewidth} \centering
   \includegraphics[width=0.95\linewidth,height=0.95\linewidth]{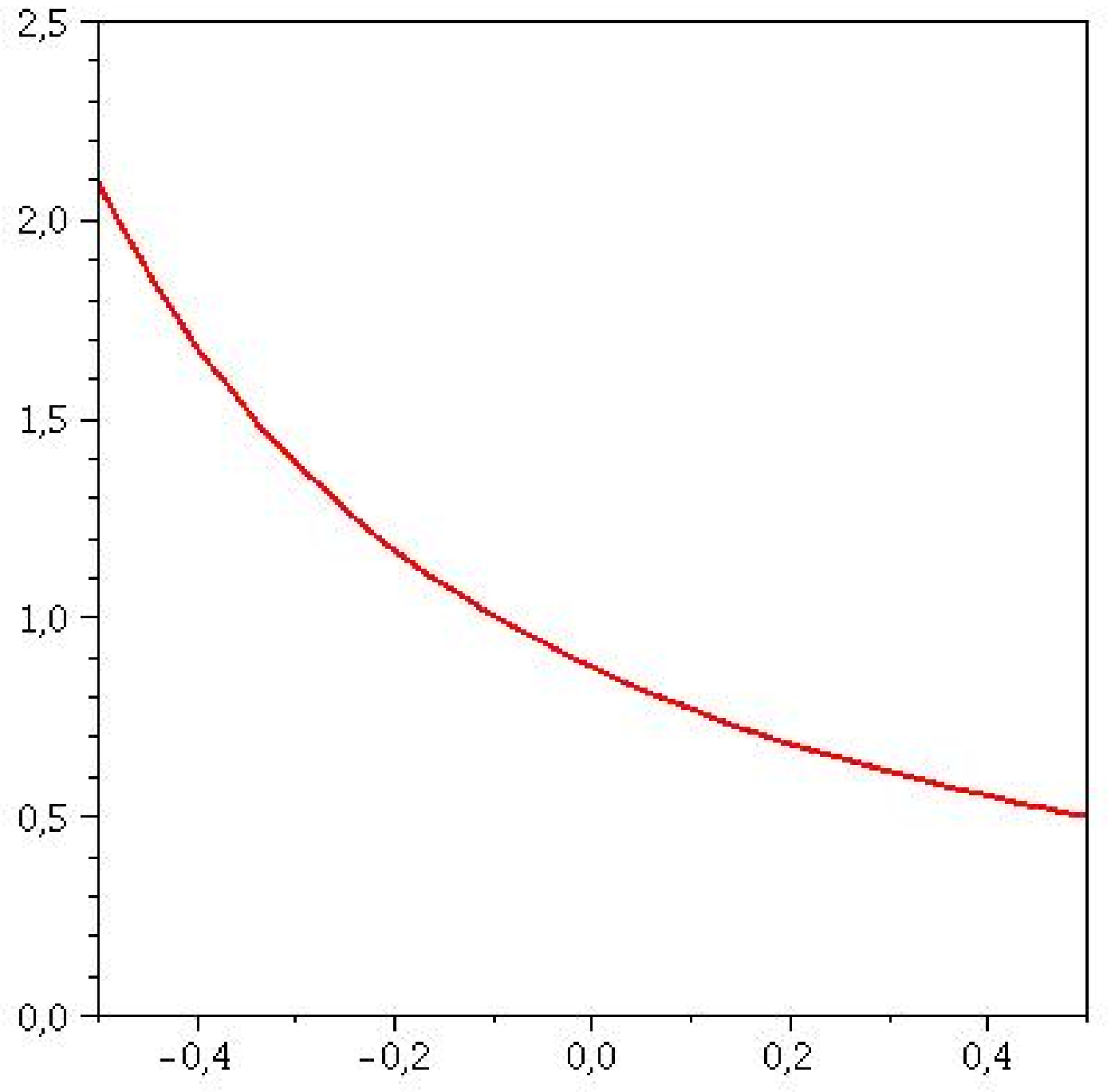}
  \end{minipage}
   \caption{{The functions $T_r$ (left), and $u_r$
(right), for $r=-\frac12$.}}
  \label{fig:Trandhr}
\end{figure}

{The heuristics of our construction is that for sufficiently
strong coupling this effect of ``polarisation'' is
reinforced and gives rise to bistable behaviour.}

\subsection{The field and the coupling}
\label{subsec:coupling} For any probability measure $Q \in \mathsf{P}(X)$, we
denote {its mean} by
\begin{equation}
  \label{eq:field-def-1} \phi (Q) := \int_{{X}} x\,
dQ(x) ,
\end{equation} {and call this} the \emph{field} of
$Q$.  With a slight abuse of notation we also write,
{for $u \in \mathcal{D}$,}
\begin{equation}
  \label{eq:field:def-2} \phi(u):= \int_{{X}} x
u(x)\,dx {\, =\phi (Q)} \quad \text{if
}u=\frac{dQ}{d\lambda} ,
\end{equation} and, {for $\x\in X^N,$}
\begin{equation}
  \label{eq:field:def-3}
\phi(\x):=\frac1N\sum_{i=1}^N{x_i}{\, =\phi (Q)}
\quad\text{if } \;Q=\frac1N\sum_{i=1}^N\delta_{x_i}\ .
\end{equation}

To define the system of globally coupled maps (both in the finite- and
the infinite{-size} case) we will, at each
step of the iteration, determine the actual parameter
{as a function of the present field. This is done
by means of a \emph{feedback function} $G:X\to R:=[-0.4,0.4]$ which we
always assume to be real-analytic
\footnote{This is only required to obtain highest regularity of the
invariant densities of the finite-size systems in Theorem \ref{theo:finite-systems}.
Everything else remains true if $G$ is merely of class $\mathcal{C}^2$.}
and \emph{S-shaped} in that it
satisfies $G'(x) > 0$ and
$G(-x)=-G(x)$ for all $x\in X$, while $G''(x) < 0$ if $x>0$. }
The most important single parameter in our
model is going to be $B:=G'(0)$ which quantifies the
\emph{coupling strength}.

\begin{remark}
  \label{remark:tanh} The following {will be our
standard example of a suitable feedback function $G$}:
  \begin{equation}
    \label{eq:G-def} G (x) := A \tanh \left(\frac{B}{A}x \right),
  \end{equation}
{where $0<A\leq0.4$ and $0\leq B\leq 18$.
(This requires some numerical effort. For $0<A\leq 0.2$ and
$0\leq B\leq 15$, elementary estimates suffice.)}
\end{remark}

{For the results to follow we shall impose a few
additional constraints on the feedback function $G$, made precise in
Assumptions I and II below. }

\subsection{The finite{-size} systems}
\label{subsec:finite-sys}

We consider a system $\Tb_N:X^N\to X^N$ of $N$ coupled copies of the
parametrised map, defined by $(\Tb_N (\x))_i = T_{r (\x)} (x_i)$ with
$r (\x) := G (\phi(\x))$. For the following theorem, which we prove in
section~\ref{sec:finite-sys-proofs}, we need the following assumption
{(satisfied by the example above)}:
\begin{equation}
  \label{eq:constraint} \eqbox{c}{\text{\textbf{Assumption
I:}}\quad\text{ $G'(x)\leq 25-50|G(x)|$ for all $x\in X$. }}
\end{equation}

\begin{theo}[\textbf{Ergodicity and mixing of
finite{-size} systems}]
\label{theo:finite-systems}
Suppose the S-shaped
function $G$ satisfies \eqref{eq:constraint}.  Then,
{for any $N \in \N$,} the map $\Tb_N:X^N\to X^N$
has a unique absolutely continuous invariant probability measure
{\boldmath$\mu$\unboldmath}$_N$.  Its density is
strictly positive and real analytic. The systems $(\Tb_N,$
{\boldmath$\mu$\unboldmath}$_N)$ are exponentially
mixing in various strong senses, in particular do H\"older observables
have exponentially decreasing correlations.
\end{theo}

The key to the proof is an estimate ensuring uniform expansion.
After establishing the latter in
Section~\ref{sec:finite-sys-proofs}, the theorem follows from
``folklore'' results whose origins are not so easy to locate in the
literature. In a $C^2$-setting, existence, uniqueness and exactness of
an invariant density were proved essentially by Krzyzewski and Szlenk
\cite{KrSz69}. Exponential mixing follows from the compactness of the
transfer operator {first} observed by Ruelle \cite{Rue76}.
For a result which applies in our situation and
entails Theorem 1, we refer to the main theorem of \cite{May84}.

\subsection{ {The self-consistent PFO and the
thermodynamic limit of the finite-size systems} }
\label{subsec:selfconsistentpfo}

Since the coupling we defined is of mean-field
type, we can adapt from the probabilistic literature (see for example
\cite{Szn91,DaGa1}) the classical method of taking the \emph{thermodynamic
limit} of our family of finite-size systems $\Tb_N$, as $N\to\infty$.
To do so, consider the set
${\mathsf{P}}(X)$ of Borel probability measures on $X$, equipped with the
topology of weak convergence and the resulting Borel $\sigma$-algebra on ${\mathsf{P}}(X)$.
Define $\widetilde T : {\mathsf{P}}(X) \to {\mathsf{P}}(X)$ by
\begin{equation}
  \label{eq:r_mu}
    \widetilde{T}(Q){:= Q \circ T_{r(Q)}^{-1}},\quad
  \text{where}\quad
  {r(Q)} := G ( {\phi}(Q)).
\end{equation}

We can then represent the evolution of any finite-size
system using $\widetilde T$. Indeed, if
$\epsilon_N(\x):=\frac1N\sum_{i=1}^N\delta_{x_i}$ is the empirical
measure of $\x={(x_i)}_{1\leq i\leq N}$, then
$\epsilon_N : X^N \to {\mathsf{P}}(X)$ satisfies
$\epsilon_N \circ \Tb_N =\widetilde{T} \circ \epsilon_N$.

Furthermore, when {restricted} to the
set of probability measures absolutely continuous with
respect to $\lambda$, $\tilde T$ is represented by the
\emph{self-consistent Perron Frobenius operator}, which is the
nonlinear positive operator $\widetilde{P}$ defined as
\begin{equation}
\widetilde{P}:L_{1}(X,\lambda)\rightarrow
L_{1}(X,\lambda)\text{,
\quad}\widetilde{P}u:=P_{G(\phi(u))}u\text{.}
\end{equation}
Clearly, this map satisfies $\widetilde
T({u}\cdot\lambda)= (\tilde{P}{u})\cdot\lambda$
and preserves the set $\mathcal{D}$ of probability densities.
Note, however, that it does not
contract, i.e. there are $u,v\in\mathcal{D}$ such that $ \|
\widetilde{P} u - \widetilde{P} v \|_{L_{1}(X,\lambda)} > \| u - v
\|_{L_{1}(X,\lambda)} $.

One may finally join these two aspects, the action of $\widetilde T$
on means of Dirac masses, or on absolutely continuous measures,
via the following observation:

\begin{prop}[\textbf{Propagation of chaos}]
  \label{pr:propchaos} Let $Q = u\cdot \lambda \in {\mathsf{P}}(X)$,
  with $u\in{\mathcal D}$. If ${(x_i)}_{i\geq1}$ is chosen according
  to $Q^{\otimes\N}$, then, for any $n\geq0$, the empirical
  measures $\epsilon_N(\Tb_N^n(x_1,\ldots,x_N))$ converge weakly to
  $({\tilde{P}}^n u)\cdot\lambda$ as $N\to\infty$.
\end{prop}

This result confirms the point of view that the self-consistent
PFO ${\tilde P}$ represents the infinite-size thermodynamic limit
$N\to\infty$ of the finite-size systems $\Tb_N$.
Its proof is reasonably simple (easier than for stochastic evolutions).
The only difficulty is that $\tilde T$ is not a continuous map on the
whole of ${\mathsf{P}}(X)$. This can be overcome with the following lemma,
which Proposition \ref{pr:propchaos} is a direct consequence of,
and whose proof is given in Section \ref{se:conTtil}.

\begin{lemma}[\textbf{Continuity of $\tilde T$ at non-atomic measures}]
\label{le:contTtil} Assume
that a sequence ${(Q_n)}_{n\geq1}$ {in ${\mathsf{P}}(X)$}
converges weakly to {some non-atomic $Q$}. Then 
${({\widetilde T}Q_n)}_{n\geq1}$
converges weakly to ${\widetilde T}Q$.
\end{lemma}

Here is an immediate consequence of this lemma that will be used below.

\begin{coro}
  \label{coro:contTtil}
  Assume that a sequence ${(\pi_n)}_{n\geq1}$ of $\widetilde T$-invariant
  Borel probability measures \emph{on} ${\mathsf{P}}(X)$
  converges weakly to some probability $\pi$ \emph{on} ${\mathsf{P}}(X)$. If
  there is a Borel set $A\subseteq{\mathsf{P}}(X)$ with $\pi(A)=1$
  which only contains non-atomic measures, then $\pi$ is also
  $\widetilde T$-invariant.
\end{coro}

\subsection{The {long-term behaviour of the} infinite{-size} system}
\label{subsec:infinite-sys}

{Our goal is to analyse the asymptotics of
$\widetilde{P}$ on $\mathcal{D}$.} Some basic features of
$\widetilde{P}$ can be understood considering {the dynamics of}
\[
{\textstyle H: \left(-\frac{2}{3},\frac{2}{3}\right)\to R:= \left[-\frac{4}{10},\frac{4}{10}\right],}\
\quad {H}(r):=G(\phi(u_{r})),
\]
which governs the action of $\widetilde{P}$ on the densities $u_{r}$
introduced in \S ~\ref{subsec:family-def}, as
\begin{equation}
\widetilde{P} u_{r}=P_{{H}(r)}u_{r}.
\end{equation}
{In studying $\widetilde{P}$, we will always presuppose the following: }
\begin{equation}
  \label{eq:S-shape}
  \textbf{Assumption II:}\ \quad\ {H}\ \text{is S-shaped}.
\end{equation}
{This assumption can be {checked numerically} for specific feedback
functions $G$, like that of Remark \ref{remark:tanh}, cf. \S \ref{ssec:infinite-sys-proofs-1}
below. By (\ref{Eq_Symmetry}), ${H}(-r)=-{H}(r)$.
{Note}, however, that $r\mapsto\phi(u_r)$ itself is not S-shaped (see Figure~\ref{fig:plots-1})
so that the S-shapedness of $G$ alone is not sufficient for that of ${H}$.}\\

{Assumption II will enter our arguments only via the following
dichotomy which it entails: either}
\begin{center}
  ${H}(r)$ has a unique fixed point at $r=0$\\
  (the {\bf{stable regime}}  with ${H}'(0)\leq1$ and $r=0$ stable),\\
  or\\
  ${H}(r)$ has exactly three fixed points $-r_*<0<r_*$\\
  (the {\bf{bistable regime}} with ${H}'(0)>1$ and $\pm r_*$ stable).
\end{center}
{We will see that  ${H}'>0$ and ${H}^{\prime}(0)={G'(0)}/6$,
so that the stable regime corresponds
to the condition $G'(0)\leq6$. Observe now that}
\begin{equation}
  \label{eq:r-star}
  \widetilde{P}u_{r}=u_{r}\quad \text{ iff } \quad
  \begin{cases}
    r=0& \text{({in the stable regime})}\\
    r\in\{0,\pm r_{\ast}\}&\text{({in the bistable regime})}
  \end{cases}
\end{equation}
(since $u_{r}\neq u_{r^{\prime}}$ for $r\neq r^{\prime}$, and
each $T_{r}$ is ergodic). We are going to show that the fixed points $u_{0}=1_{X}
$, and $u_{\pm r_{\ast}}$ dominate the long-term behaviour of $\widetilde{P} $
on $\mathcal{D}$ completely, {and that they inherit the stability
properties of the corresponding parameters $-r_*<0<r_*$. Therefore, the stable/bistable
terminology {for $H$} introduced above also provides an appropriate
description of the asymptotic
behaviour of $\widetilde{P}$.}

\begin{theo} [\textbf{Long-term behaviour of }$\widetilde{P}$\textbf{\ on
  }$\mathcal{D}$ ]
\label{theo:infinite-system}
Consider $\widetilde{P}:\Dc\rightarrow\Dc$, $\Dc$ equipped
with the metric inherited from $L_1(X,\lambda)$.
{Assuming (I) and (II), we have the following:}

\begin{enumerate}[1)]
\item \label{theo2:item1}
  {In the stable regime}, $u_0$ is the unique fixed point of $\widetilde{P}$,
  and attracts all densities, that is,
  $$\lim_{n\to\infty}\widetilde{P}^n u=u_0\ \quad \text{for all}\ u\in\Dc.$$
\item {In the bistable regime}, $\{u_{-r_{\ast}},u_0,u_{r_{\ast}}\}$
  are the only fixed points of $\widetilde{P}$. Now $u_0$ is unstable, while
  $u_{-r_{\ast}}$ and $u_{r_{\ast}}$ are stable. More precisely:
  \begin{enumerate}[a)]
  \item $u_{\pm r_{\ast}}$ are stable fixed points for $\widetilde{P}$ in the
    sense that
      their respective basins of attraction are $L_1$-open.
  \item If $u {\in \mathcal{D}}$ is not attracted by
    $u_{-r_{\ast}}$ or $u_{r_{\ast}}$, then it is attracted by $u_{0}$.
  \item $u_{0}$ is not stable. Indeed, $u_0$ can be $L_1$-approximated
    by convex analytic densities from either basin. It is a hyperbolic fixed
    point of $\widetilde P$ in the sense made precise in
    Proposition~\ref{prop:hyperbolic-fixed} of Section~\ref{sec:infinite-sys-proofs}.
  \end{enumerate}
\end{enumerate}
\end{theo}

{
\begin{exmp}
In case $G(x)=A \tanh(Bx/A)$ with $0<A\leq0.4$ and $0\leq B\leq 18$,
both theorems apply.
The infinite-size system is stable iff $B\leq6$, and bistable otherwise,
while all finite-size systems have a unique a.c.i.m. in
this parameter region.
\end{exmp}
}

The theorem summarises the contents of Propositions \ref{prop:summary2},
\ref{prop:basin-boundary} and~\ref{prop:hyperbolic-fixed} of
Section~\ref{sec:infinite-sys-proofs} (which, in fact, provide more detailed
information).  The proofs rest on the fact that PFOs of maps with full
fractional-linear branches leave the class of Herglotz-Pick-Nevanlinna
functions invariant. This observation can be used to study the action of
$\widetilde{P}$ in terms of an iterated function system on the interval
$[-2,2]$ with two fractional-linear branches and place dependent
probabilities. {In the bistable regime} the system is of course not
contractive, but it has strong monotonicity properties and
special geometric features which allow to prove the theorem.\\

Our third theorem, which is essentially a corollary to the previous ones,
describes the passage from finite-size systems to the infinite-size system.
Below, weak convergence of the \boldmath$\mu$\unboldmath$_N \in
{\mathsf{P}}(X^N)$ to some \boldmath$\mu\,$\unboldmath $\in
{\mathsf{P}}(X^\N)$ means that $\int \varphi\,d$\boldmath$\mu$\unboldmath$_N
\to \int \varphi\,d$\boldmath$\mu$\unboldmath \, for all continuous $\varphi :
X^\N \to \R$ which only depend on finitely many coordinates. (So that $\int
\varphi\,d$\boldmath$\mu$\unboldmath$_N$ is defined, in the obvious fashion,
for $N$ large enough.)

\begin{theo}[\textbf{From finite to infinite size -- the limit as $N\to\infty$}]
  \label{theo:passage-to-infinity}
  The $\Tb_N$-invariant probability measures \boldmath$\mu$\unboldmath$_N$ of
  Theorem~\ref{theo:finite-systems} correspond to the $\widetilde T$-invariant
  probability measures \boldmath$\mu$\unboldmath$_N\circ\epsilon_N^{-1}$ on
  ${\mathsf{P}}(X)$. All weak accumulation points $\pi$ of the
  latter sequence are ${\widetilde T}$-invariant probability measures concentrated
  on the set of measures absolutely continuous
  w.r.t. $\lambda$. Furthermore:
  \begin{enumerate}[1)]
  \item In the stable regime, the sequence
    $($\boldmath$\mu$\unboldmath$_N\circ\epsilon_N^{-1})_{N\geq1}$ converges
    weakly to the point mass $\delta_{\lambda}$. In other words, the sequence
    $($\boldmath$\mu$\unboldmath$_N)_{N\geq1}$ converges
    weakly to the pure product measure $\lambda^\N$ on $X^\N$.
  \item In the bistable regime, each weak accumulation point $\pi$ of the sequence
    $($\boldmath$\mu$\unboldmath$_N\circ\epsilon_N^{-1})_{N\geq1}$ is of the
    form
    $\alpha\,\delta_{u_{-r_*}\lambda}+(1-2\alpha)\,\delta_{u_0\lambda}+\alpha\,\delta_{u_{r_*}\lambda}$
    for some $\alpha\in[0,\frac12]$. In other words, each weak accumulation point of
    the sequence $($\boldmath$\mu$\unboldmath$_N)_{N\geq1}$ is of the form
    $\alpha(u_{-r_*}\lambda)^\N+(1-2\alpha)\lambda^\N+\alpha(u_{r_*}\lambda)^\N$.
  \end{enumerate}
\end{theo}
\begin{remark}
  We cannot prove, so far, that $\alpha=\frac12$, which is to be expected
  because $u_0$ is an unstable fixed point of $\widetilde P$. In
  Section~\ref{sec:noisy} we show that $\alpha=\frac12$ indeed, if some
  small noise is added to the system.
\end{remark}
\begin{proof}[Proof of Theorem~\ref{theo:passage-to-infinity}]
  As $\epsilon_N \circ \Tb_N =\widetilde{T} \circ \epsilon_N$, the
  $\Tb_N$-invariant probability measures \boldmath$\mu$\unboldmath$_N$ of
  Theorem~\ref{theo:finite-systems} correspond to $\widetilde T$-invariant
  probability measures \boldmath$\mu$\unboldmath$_N\circ\epsilon_N^{-1}$ on
  ${\mathcal P}(X)$. Their possible weak accumulation points are all
  concentrated on sets of measures from ${\mathcal P}(X)$ with density
  w.r.t. $\lambda$, see Theorem~3 in \cite{Keller00}. (The proof of that part
  of the theorem we refer to does not rely on the continuity of the local maps
  that is assumed in that paper.) Therefore Corollary~\ref{coro:contTtil}
  shows that all these accumulation points are ${\widetilde T}$-invariant
  probability measures concentrated on measures with density
  w.r.t. $\lambda$. In other words, they can be interpreted as $\widetilde
  P$-invariant probability measures on $\Dc$. Now
  Theorem~\ref{theo:infinite-system} implies that the sequence
  $($\boldmath$\mu$\unboldmath$_N\circ\epsilon_N^{-1})_{N\geq1}$ converges
  weakly to the point mass $\delta_{u_0\lambda}$ in the stable regime,
  whereas, in the bistable regime, each such limit measure is of the form
  $\alpha\,\delta_{u_{-r_*}\lambda}+(1-2\alpha)\,\delta_{u_0\lambda}+\alpha\,\delta_{u_{r_*}\lambda}$
  for some $\alpha\in[0,\frac12]$ (observe the symmetry of the system). Now
  the corresponding assertions on the measures \boldmath$\mu$\unboldmath$_N$
  follow along known lines, for a reference see
  e.g. \cite[Proposition~1]{Keller00}.
\end{proof}


\section{Proofs: the finite{-size} systems}
\label{sec:finite-sys-proofs}
We assume throughout this section that
\begin{equation}
  \label{eq:finite-ass}
  |G(x)|\leq0.5\text{ and }G'(x)\leq 25-50|G(x)|\text{
  for all }|x|\leq\frac12.
\end{equation}
In order to apply the main theorem of Mayer \cite{May84} we must check his
assumptions (A1) -- (A4) for the map $\Tb=\Tb_N$. To that end define
$\Fb:X^N\to[-\frac12,\frac32]^N$ by $(\Fb(\x))_i=f_{M_{r(\x)}}(x_i)$.  Obviously
$\Tb(\x)=\Fb(\x)\text{ mod }\left(\mathbb{Z}+\frac{1}{2}\right)^N$, and (A1) --
(A4) follow readily from the following facts that we are going to prove:

\begin{lemma}
  \label{lemma:homeo}
  $\Fb:X^N\to[-\frac12,\frac32]^N$ is a homeomorphism which extends to a
  diffeomorphism between open neighbourhoods of $X^N$ and
  $[-\frac12,\frac32]^N$.
\end{lemma}

\begin{lemma}
\label{lemma:analytic-diffeo}
  The inverse $\Fb^{-1}$ of $\Fb$ is
  real analytic and can be continued to a holomorphic mapping on a
  complex $\delta$-neighbourhood $\Omega$ of $[-\frac12,\frac32]^N$ such that
  $\Tb^{-1}(\Omega)$ is contained in a $\delta'$-neighbourhood of $X^N$ for
  some $0<\delta'<\delta$.
\end{lemma}

To verify these two lemmas we need the following uniform expansion estimate
which we will prove at the end of this section.  (Here $\|.\|$ denotes the
Euclidean norm.)

\begin{lemma}[\textbf{Uniform expansion}]
  \label{lemma:inv-contraction}
  There is a constant $\rho\in(0,1)$ such
  that $\|(D\Fb(\x))^{-1}\|\leq\rho$ for all $N\in\N$ and $\x\in X^N$.
\end{lemma}

\begin{proof}[Proof of Lemma~\ref{lemma:homeo}]
  Obviously $\Fb(X^N)\subseteq[-\frac12,\frac32]^N$. Hence it is sufficient to
  prove the assertions of the lemma for the map
  $\widetilde\Fb:=\frac12(\Fb-(\frac12,\dots,\frac12)^\transpose):X^N\to X^N$.
  As each $f_{M_r}$ is differentiable on $(-1,1)$ (recall that $|r|<\frac23$),
  $\widetilde\Fb$ extends to an analytic mapping from $(-1,1)^N\to\R^N$.  By
  Lemma~\ref{lemma:inv-contraction}, it is locally invertible on
  $\Omega_\varepsilon:=(-\frac12-\varepsilon,\frac12+\varepsilon)^N$ for each
  sufficiently small $\varepsilon\geq0$.  (Note that $\Omega_0=\interior{X}$.)
  All we need to show is that this implies global invertibility of
  $\widetilde\Fb|_{\Omega_0}:\Omega_0\to\Omega_0$, because then the
  possibility to extend $\widetilde\Fb$ diffeomorphically to a small open
  neighbourhood of $X^N$ in ${\mathbb R}^N$ follows again from the local
  invertibility on $\Omega_\varepsilon$ for some $\varepsilon>0$.

  So we prove the global invertibility of
  $\widetilde\Fb|_{\Omega_0}:\Omega_0\to\Omega_0$.  As each $\tilde
  f_{M_r}:=\frac12(f_{M_r}-\frac12):X\to X$ is a homeomorphism that leaves
  fixed the endpoints of the interval $X$, we have $\widetilde\Fb(\partial
  X^N)\subseteq\partial X^N$ and $\widetilde\Fb(\Omega_0)\subseteq\Omega_0$.
  Observing the simple fact that $\Omega_0$ is a paracompact connected smooth
  manifold without boundary and with trivial fundamental group, we only need
  to show that $\widetilde\Fb|_{\Omega_0}:\Omega_0\to\Omega_0$ is proper in
  order to deduce from \cite[Corollary 1]{Chi98} that
  $\widetilde\Fb|_{\Omega_0}$ is a diffeomorphism of $\Omega_0$. So let $K$ be
  a compact subset of $\Omega_0$.  As $\widetilde\Fb|_{\Omega_0}$ extends to
  the continuous map $\widetilde\Fb:X^N\to X^N$ and as $\widetilde\Fb(\partial
  X^N)\subseteq\partial X^N$, the set
  $\widetilde\Fb|_{\Omega_0}^{-1}(K)=\widetilde\Fb^{-1}(K)\subset\interior{X^N}\subseteq
  X^N$ is closed and hence compact. Therefore
  $\widetilde\Fb|_{\Omega_0}:\Omega_0\to\Omega_0$ is indeed proper.
\end{proof}

\begin{proof}[Proof of Lemma~\ref{lemma:analytic-diffeo}]
  As $\Fb$ is real analytic on a real neighbourhood of $X^N$, the real
  analyticity of $\Fb^{-1}$ on a real neighbourhood of $[-\frac12,\frac32]^N$
  follows from the real analytic inverse function theorem \cite[Theorem
  18.1]{KrPa92}. It extends to a holomorphic function on a complex
  $\delta$-neighbourhood $\Omega$ of $[-\frac12,\frac32]^N$ -- see e.g. the
  discussion of complexifications of real analytic maps
  in\cite[pp.162-163]{KrPa92}. If $\delta>0$ is sufficiently small,
  Lemma~\ref{lemma:inv-contraction} implies that $\Fb^{-1}$ is a uniform
  contraction on $\Omega$. Hence the $\delta'$ in the statement of
  the lemma can be chosen strictly smaller than $\delta$.
\end{proof}

\begin{proof}[Proof of Lemma~\ref{lemma:inv-contraction}]
  Recall that $(\Fb(\x))_i=f_{{r(\x)}}(x_i)$ where $r(\x)=G(\phi(\x))$,
  $\phi(\x)=\frac{1}{N} \sum_{i = 1}^N x_i$ and we write $f_r$ instead of
  $f_{M_r}$. Denote $g (\x) = G' (\phi (\x))$,
\begin{eqnarray*}
  \Delta_1 (\x) & \assign & \diag (f_r' (x_1), \ldots , f_r'
  (x_N))\\
  \Delta_2 (\x) & \assign & \diag ( \frac{\partial
  f_r}{\partial r} (x_1), \ldots , \frac{\partial f_r}{\partial r} (x_N))\\
  E_N & \assign & \left(\begin{array}{c}
    \frac{1}{N} \ldots  \frac{1}{N}\\
    \vdots\quad\quad \vdots\\
    \frac{1}{N} \ldots  \frac{1}{N}
  \end{array}\right)
\end{eqnarray*}
and observe that $q (x) \assign (4 - r^2) \frac{\partial f_r}{\partial r} (x)
/ f_r' (x)$ simplifies to $q(x)=1-4x^2$
so that
\begin{displaymath}
  \Delta_1(\x)^{- 1} \Delta_2(\x) = \frac{1}{4 - r^2} \diag \left( q (x_1),
   \ldots, q (x_N \right)) = : \frac{1}{4 - r^2} \Delta_3 (\x)\ .
\end{displaymath}
Then the derivative of the coupled map $\Fb(\x)=(f_{r (\x)} (x_1),
\ldots , f_{r (\x)} (x_N))$ is
\begin{displaymath}
  D\Fb(\x) = \Delta_1 (\x) + \Delta_2 (\x) E_N g (\x) = \Delta_1 (\x) \left( {\bf 1} +
   \frac{1}{4 - r^2} \Delta_3 (\x) E_N g (\x) \right)\ ,
\end{displaymath}
with $\bf{1}$ denoting the identity matrix. Letting
\begin{eqnarray*}
  \mathbf{q}=\mathbf{q}(\x) & \assign & (q (x_1), \ldots ., q (x_N))^\transpose\\
  \mathbf{e}_N & \assign & ( \frac{1}{N}, \ldots ., \frac{1}{N})^\transpose
\end{eqnarray*}
and observing that $\Delta_3 (\x) = \diag (\mathbf{q}(\x))$ so that
$\Delta_3(\x)E_N=\mathbf{q}\,\mathbf{e}_N^\transpose$, the inverse of $D\Fb
(\x)$ is
\begin{displaymath}
  D \Fb(\x)^{- 1}
  =
  \left({\bf 1} - \frac{g (\x)}{4 - r^2 + g (\x)\mathbf{e}_N^\transpose
      \mathbf{q}(\x)} \mathbf{q}(\x)\mathbf{e}_N^\transpose \right)
  \Delta_1 (\x)^{- 1}  \ .
\end{displaymath}
In order to check conditions under which $\Fb$ is uniformly expanding in all
directions, it is sufficient to find conditions under which $\|D\Fb(\x)^{-1}\|<
1$ \emph{uniformly} in $\x$. Observe first that $\| \Delta_1 (\x)^{- 1} \|
\leq (\inf f_r')^{- 1} \leq \frac{1}{2} \frac{2 + |r|}{2 - | r|} <
1$ for $|r| < \frac{2}{3}$. From now on we fix a point $\x$ and suppress it as
an argument to all functions. Then, if $\mathbf{v}$ is any vector in $\R^N$,
some scalar multiple of it can be decomposed in a unique way as $\alpha
\mathbf{v} =\mathbf{q}- \mathbf{p}$ where $\mathbf{p}$ is perpendicular to
$\mathbf{q}$. Denote $\bar{p} = \mathbf{e}_N^\transpose \mathbf{p}$, $\bar{q} =
\mathbf{e}_N^\transpose \mathbf{q}$, and observe that $\bar{q} \geq 0$ as
$\mathbf{q}$ has only nonnegative entries. We estimate the euclidian norm of \
$({\bf 1} - \frac g{4 - r^2 + g \mathbf{e}^\transpose_N \mathbf{q}} \mathbf{q}\,
\mathbf{e}^\transpose_N) (\alpha \mathbf{v})$:
\begin{equation}
  \label{eq:norm-estimate-1}
  \begin{split}
    &\left\| \left({\bf 1} - \frac{g}{4 - r^2 + g \mathbf{e}^\transpose_N
            \mathbf{q}} \mathbf{q}\, \mathbf{e}^\transpose_N \right) (\alpha
        \mathbf{v})
      \right\|^2\\
    =& \left( 1 + \frac{g \bar{p} - g \bar{q}}{4 - r^2 + g \mathbf{} \bar{q}}
    \right)^2 \| \mathbf{q} \|^2 +\| \mathbf{p} \|^2
    \\
    =& \left( 1 + \frac{ \bar{p} - \bar{q}}{\Gamma + \bar{q}} \right)^2 \|
    \mathbf{q} \|^2 +\| \mathbf{p} \|^2 \\
    = & \left( \frac{1 + \Gamma^{- 1} \bar{p}}{1 + \Gamma^{- 1} \bar{q}}
    \right)^2 \| \mathbf{q} \|^2 +\| \mathbf{p} \|^2
  \end{split}
\end{equation}
where $\Gamma \assign \frac{4 - r^2}{g}$. As $\bar{p} \leq N^{- 1 / 2} \|
\mathbf{p} \|$ and $\bar{q} \geq N^{- 1} \| \mathbf{q} \|^2$ (observe
that all entries of $\mathbf{q}$ are bounded by $1$), we can continue the
above estimate with
\begin{displaymath}
  \leq  \| \mathbf{q} \|^2 +\| \mathbf{p} \|^2 + 2\| \mathbf{p}
  \|\| \mathbf{q} \| \frac{\Gamma^{- 1} N^{- 1 / 2} \| \mathbf{q} \|}{(1 +
  \Gamma^{- 1} N^{- 1} \| \mathbf{q} \|^2)^2} +\| \mathbf{p} \|^2
  \frac{\Gamma^{- 2} N^{- 1} \| \mathbf{q} \|^2}{(1 + \Gamma^{- 1} N^{- 1}
  \| \mathbf{q} \|^2)^2 }
\end{displaymath}
To estimate this expression we abbreviate temporarily $t\assign N^{- 1 / 2} \|
\mathbf{q} \|$. Then $0 \leq t \leq 1$, and straightforward
maximisation yields:
\begin{align*}
  \frac{\Gamma^{- 1} N^{- 1 / 2} \| \mathbf{q} \|}{(1 + \Gamma^{- 1} N^{- 1}
  \| \mathbf{q} \|^2)^2} & \leq \frac{9}{16 \sqrt{3}
  \sqrt{\Gamma}}\\
  \frac{\Gamma^{- 2} N^{- 1} \| \mathbf{q} \|^2}{(1 + \Gamma^{- 1} N^{- 1}
  \| \mathbf{q} \|^2)^2 } & \leq \frac{1}{4 \Gamma}
\end{align*}
So we can continue
the above estimate by
\begin{displaymath}
  \text{(\ref{eq:norm-estimate-1})}
  \leq  (\| \mathbf{q} \|^2 +\| \mathbf{p} \|^2) \left( 1 +
  \frac{9}{16 \sqrt{3}  \sqrt{\Gamma}} + \frac{1}{4 \Gamma} \right)
\end{displaymath}
where $\| \mathbf{q} \|^2 +\| \mathbf{p} \|^2 =\| \alpha \mathbf{v} \|^2 .$
Hence
\begin{displaymath}
  \left\| \left({\bf 1} - \frac{g}{1 - r + g \mathbf{e}^\transpose_N \mathbf{q}}
  \mathbf{q} \mathbf{e}^\transpose_N \right) (\alpha \mathbf{v}) \right\|
   \leq  \left( 1 + \frac{9}{16 \sqrt{3}  \sqrt{\Gamma}} + \frac{1}{4
  \Gamma} \right)^{1 / 2}
  \|\alpha{\bf v}\|
\end{displaymath}
and therefore
\begin{displaymath}
  \|D\Fb^{- 1} (x)\|  \leq \rho \assign \left( \frac{1}{2} \cdot
  \frac{2 + |r|}{2 - |r|} \right)  \left( 1 + \frac{9}{16 \sqrt{3}
  \sqrt{\Gamma}} + \frac{1}{4 \Gamma} \right)^{1 / 2}\ .
\end{displaymath}
Observing $\Gamma = \frac{4 - r^2}{g}$ one finds numerically that the norm is
bounded by $0.99396$ uniformly for all $x$, if $-0.5\leq r \leq 0.5$ and
$0\leq g \leq25-50r$. Hence the map $\Fb$ is uniformly expanding in all
directions provided (\ref{eq:finite-ass}) holds.
\end{proof}

\section{{An iterated function system representation for smooth densities}}
\label{sec:ifs}

\subsection{An invariant class of densities}

The PFOs $P_{r}$ allow a detailed analysis since their action on certain
densities has a convenient explicit description: Consider the family
$(w_{y})_{y\in(-2,2)}$ of probability densities on $X$ given by
\[
w_{y}(x):=\frac{1-y^{2}/4}{(1-xy)^{2}}\text{, \quad}x\in X\text{.}%
\]
As pointed out in \cite{Sch81} (using different parametrisations),
Perron-Frobenius operators $P_{f_{M}}$ of fractional-linear maps $f_{M}$ act
on these densities via their \emph{duals} $f_{M^{\#}}$, where
\[
M^{\#}:=\left(
\begin{array}
[c]{cc}%
0 & 1\\
1 & 0
\end{array}
\right)  \cdot M\cdot\left(
\begin{array}
[c]{cc}%
0 & 1\\
1 & 0
\end{array}
\right)  =\left(
\begin{array}
[c]{cc}%
d & c\\
b & a
\end{array}
\right)  \text{ \quad for }M=\left(
\begin{array}
[c]{cc}%
a & b\\
c & d
\end{array}
\right)  \text{,}%
\]
in that
\begin{equation}
\label{eq:transfer-by-fract-lin}
P_{f_{M}}(1_{J}\cdot w_{y})=\left(  \int_{J}w_{y}\,d\lambda\right)  \cdot
w_{y^{\prime}}\text{ \quad with }y^{\prime}=f_{M^{\#}}(y)\text{,}%
\end{equation}
for matrices $M$ and intervals $J\subseteq X$ for which $f_{M}(J)=X$. (This
can also be verified by direct calculation).
Since $f_{{0\,1\choose1\,0}}(x)=\frac1x$, one can compute
the duals
\begin{equation}
\label{eq:sigma-tau-def}
  \begin{split}
    \sigma_{r}(y)
    &:=
    f_{M_{r}^{\#}}(y) =\frac1{f_{M_r}(\frac1 y)}
    =
    \frac{2(y+r)}{(r+1)y+r+4}\text{ \quad and}\\
    \tau_{r}(y)
    &:=
    f_{N_{r}^{\#}}(y)
    =
    \frac1{f_{N_r}(\frac1  y)}
    =
    \frac{\sigma_{r}(y)}{1-\sigma_{r}(y)} =\frac{2(y+r)}{(r-1)y-r+4}%
  \end{split}
\end{equation}
of the individual branches of $T_{r}$, then express $P_{r}\,w_{y}$ as the
convex combination
\begin{align}
P_{r}\,w_{y}  & =P_{f_{M_{r}}}(1_{[-\frac{1}{2},\alpha_{r})}\cdot
w_{y})+P_{f_{N_{r}}}(1_{(\alpha_{r},\frac{1}{2}]}\cdot w_{y})\nonumber\\
& =p_{r}(y)\cdot w_{\sigma_{r}(y)}+(1-p_{r}(y))\cdot w_{\tau_{r}(y)}%
\text{,}\label{Eq_Pr_gy_in_terms_of_sigma_tau}%
\end{align}
with weights
\begin{equation}
  \label{eq:p_r}
  p_{r}(y):=\int_{-1/2}^{\alpha_{r}}w_{y}(x)\,dx=\frac12-\frac{r+y}{4+ry}\quad
  \text{and}\quad1-p_{r}(y)=\frac12+\frac{r+y}{4+ry}
\end{equation}

It is straightforward to check that for every $r\in(-2,2)$ the
functions $\sigma_{r}$, $\tau_{r}$ are continuous and strictly increasing on
$[-2,2]$ with images $\sigma_{r}([-2,2])=[-2,2/3]$ and $\tau_{r}%
([-2,2])=[-2/3,2]$.

{}From this remark and (\ref{Eq_Pr_gy_in_terms_of_sigma_tau}) it is clear that the $P_{r}$
preserve the class of those $u\in\mathcal{D}$ which are convex combinations
$u=\int_{(-2,2)}w_{y}\,d\mu(y)=\int w_{\bullet}\,d\mu$ of the special densities
$w_{y}$ for some \emph{representing measure} $\mu$ from $\mathsf{P(-2,2)}$. We find
that $P_{r}$ acts on representing measures according to
\begin{gather}
 P_{r}\,\left(  \int_{{(-2,2)}} w_{\bullet}\,d\mu\right)
 =\int_{{(-2,2)}} w_{\bullet}\,d\left(  \mathcal{L}%
_{r}^{\ast}\mu\right)  \text{ \quad with}%
\label{Eq_Pr_acting_on_representing_measures}\\
\mathcal{L}_{r}^{\ast}\mu:=(p_{r}\cdot\mu)\circ\sigma_{r}^{-1}+((1-p_{r}%
)\cdot\mu)\circ\tau_{r}^{-1}\text{,}\nonumber
\end{gather}
where $p_{r}\cdot\mu$ denotes the measure with density $p_{r}$ w.r.t. $\mu$.

To continue, we need to collect several facts about the dual maps $\sigma_{r}$
and $\tau_{r}$.  We have
\begin{equation}
  \label{eq:sigma_prime}
\sigma_{r}^{\prime}(y)=\frac{2(4-r^{2})}{(ry+y+r+4)^{2}}\text{ \quad and
}\quad\tau_{r}^{\prime}(y)=\frac{2(4-r^{2})}{(ry-y-r+4)^{2}}\text{,}%
\end{equation}
showing that $\sigma_{r}$ and $\tau_{r}$ are strictly concave, respectively
convex, on $[-2,2]$. One next gets readily from \eqref{eq:sigma-tau-def} that
$1/\sigma_{r}-1/\tau_{r}=1$ wherever defined, and
\begin{equation*}
\sigma_{r}(y)<\tau_{r}(y)\text{ for }y\in\lbrack-2,2]\diagdown\{-r\}
\end{equation*}
while (and this observation will be crucial later on) $\sigma_r$ and $\tau_r$ have a
common zero $z_r:=-r$ and
\begin{equation}
  \label{eq:z_r}
  \sigma_r'(z_r)=\tau_r'(z_r)=\frac{2}{4-r^2},\quad\text{and}\quad
  \sigma_r(z_r+t)=\frac{2\,t}{(r+1)\,t+4-{r}^{2}} \; .
\end{equation}

\vspace{0.3cm}%

In the following, we restrict our parameter $r$ to the set $R=\left[
-\frac{4}{10},\frac{4}{10}\right]  $. Direct calculation proves that, letting
$Y:=\left[  -\frac{2}{3},\frac{2}{3}\right]  $, we have
\begin{equation}
\sigma_{r}(Y)\cup\tau_{r}(Y)\subseteq Y\text{ \quad if }r\in R\text{,
}\label{Eq_Invariance_of_Y}%
\end{equation}
so that $Y$ is an invariant set for all such $\sigma_{r}$ and $\tau_{r}$, and
that
\begin{equation}
  \label{eq:less3/4}
  \sup_{Y}\sigma_{r}^{\prime}=\sigma_{r}^{\prime}\left(  -\frac{2}{3}\right)
  \leq\frac{3}{4}\text{ \quad and \quad}\sup_{Y}\tau_{r}^{\prime}=\tau
  _{r}^{\prime}\left(  \frac{2}{3}\right)  \leq\frac{3}{4}\text{ \quad for }r\in
  R\text{,}%
\end{equation}
which provides us with a common contraction rate on $Y$ for the $\sigma_{r}$
and $\tau_{r}$ from this parameter region. All these features of $\sigma_r$
and $\tau_r$ are
illustrated by Figure~\ref{fig:sigma-tau}.
\begin{figure}
  \begin{minipage}{0.49\linewidth}
  \centering
    \includegraphics[width=0.95\linewidth,height=0.95\linewidth]{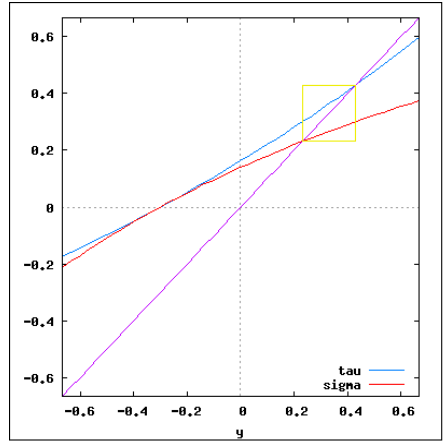}
  \end{minipage}
  \caption{The dual maps $\sigma_r$ and $\tau_r$ for $r=0.3$. The small
    invariant box has endpoints $\gamma_r$ and $\tau_r$.}
  \label{fig:sigma-tau}
\end{figure}

We denote by $\overline{w}(y):=\phi(w_{y})$
the field of the density $w_y$, and find by explicit integration that
\begin{equation}
  \label{eq:field_w_y}
  \overline{w}(y)=\left(\frac14-\frac1{y^2}\right)\log\frac{1+y/2}{1-y/2}+\frac1y
  =
  \sum_{k=0}^\infty\frac1{(2k+1)(2k+3)}\left(\frac y2\right)^{2k+1}
\end{equation}
for $y \in Y$. In particular, $\overline{w}(0)=0$ and $\overline{w}'(y) \geq \frac16 >0$,
so the field depends monotonically
on $y$. Note also that we have, for all $\mu\in\mathsf{P}(Y)$,
\begin{equation}
  \label{Eq_mean_for_convex_combi}
  \phi\left(  \int_{{Y}} w_{\bullet}\,d\mu\right)  =
  \int_{{Y}}\int_X x w_y(x)\,dx\,d\mu(y)
  =
  \int_{{Y}} \overline{w}\,d\mu .
\end{equation}

We focus on densities $u$ with representing measure $\mu$ supported on $Y$,
i.e. on the class $\mathcal{D}^{\prime}:=\{u=\int_{Y}w_{\bullet}\,d\mu$ : $\mu
\in\mathsf{P}(Y)\}$. Writing
\begin{equation}
  r_{\mu}:=
  G\left(\int_{Y}\overline w\,d\mu\right)=G(\phi(u)) \in R,
\end{equation}
and recalling (\ref{Eq_Pr_acting_on_representing_measures}),
we find that our nonlinear operator $\widetilde{P}$ acts on the representing
measures from $\mathsf{P}(Y)$ via
\begin{gather}
\widetilde{P}\,\left(  \int_{{Y}} w_{\bullet}\,d\mu\right)
=\int_{{Y}} w_{\bullet}\,d\left(
\widetilde{\mathcal{L}}^{\ast}\mu\right)  \text{ \quad with }%
\label{Eq_P_tilda_acting_on_representing_measures}\\
\widetilde{\mathcal{L}}^{\ast}\mu:=\mathcal{L}_{r_{\mu}}^{\ast}\,\mu
=(p_{r_{\mu}}\cdot\mu)\circ\sigma_{r_{\mu}}^{-1}+((1-p_{r_{\mu}})\cdot
\mu)\circ\tau_{r_{\mu}}^{-1}\text{.}\nonumber
\end{gather}
As \textrm{supp}$(\widetilde{\mathcal{L}}^{\ast}\mu)$, the support of
$\widetilde{\mathcal{L}}^{\ast}\mu$, is contained in $\sigma_{r}%
(\mathrm{supp}(\mu))\cup\tau_{r}(\mathrm{supp}(\mu))$, it is immediate from
(\ref{Eq_Invariance_of_Y}) that
\[
\widetilde{P}\mathcal{D}^{\prime}\subseteq\mathcal{D}^{\prime}\text{.}%
\]

{}For $r\in R=[-\frac4{10},\frac4{10}]$ we find that $\sigma_{r}$ and $\tau_{r}$
each have a unique stable fixed point in $Y$, given by
\[
\sigma_{r}(\gamma_{r})=\gamma_{r}:=\frac{r}{r+1}\text{ \quad and \quad}%
\tau_{r}(\delta_{r})=\delta_{r}:=\frac{-r}{r-1}\text{,}%
\]
respectively. Note that the interval $Y_{r}:=[\gamma_{r},\delta_{r}]$ of width
$2r^{2}/(1-r^{2})$ between these stable fixed points is invariant under both
$\sigma_{r}$ and $\tau_{r}$, see the small boxed region in Figure~\ref{fig:sigma-tau}. Furthermore, each of $\gamma_{r}$ and $\delta
_{r}$ is mapped to $r$ under the branch not fixing it, i.e.
\[
\sigma_{r}(\delta_{r})=\tau_{r}(\gamma_{r})=r\text{,}%
\]
meaning that, restricted to $Y_{r}$, $\sigma_{r}$ and $\tau_{r}$ are the
inverse branches of some 2-to-1 piecewise fractional linear map $S_{r}%
:Y_{r}\rightarrow Y_{r}$.

{The explicit $T_{r}$-invariant densities $u_r$ from (\ref{eq:inv-dens-2})
can be represented as $u_{r}=\int_Y w_{\bullet}\,d\mu_{r}$
with $\mu_{r}\in\mathsf{P}(Y_{r})\subseteq\mathsf{P}(Y)$ given by}
\begin{equation}
  \label{eq:mu-r-density}
  {
  \frac{d\mu_{r}}{d\lambda}(y) =
  \left(\log\frac{r^{2}-4}{9r^{2}-4}\right)^{-1}
  \frac{1_{Y_r}(y)}{1-y^2/4}.
  }
\end{equation}

{Our goal in this section is to study the asymptotic
behaviour of $\widetilde{P}$ on $\mathcal{D}'$, using its representation
by means of the IFS $\tilde\Lc^*$. We will prove}

\begin{prop} [ {\textbf{Long-term behaviour of }$\widetilde{P}$\textbf{\ on
  }$\mathcal{D}'$} ]
  \label{prop:summary}
  Take any $u \in \mathcal{D}'$, $u=\int_Y w_y\,d\mu(y)$ for some $\mu\in\mathsf{P}(Y)$.
  The following is an
  exhaustive list of possibilities for the asymptotic behaviour of
  the sequence $(\tilde\Lc^{*n}\mu)_{n\geq0}$:
  \begin{enumerate}
  \item $\tilde\Lc^{*n}\mu\succ\delta_0$ for some $n\geq0$. Then
    $(\tilde\Lc^{*n}\mu)_{n\geq0}$ converges to $\mu_{r_*}$ and hence $\tilde
    P^n u$ converges to $u_{r_*}$ in $L_{1}({X},\lambda)$.
  \item\label{item:summary-2}
    The interval $\conv(\supp(\tilde\Lc^{*n}\mu))$ contains $0$ for all
    $n\geq0$. Then $(\tilde\Lc^{*n}\mu)_{n\geq0}$ converges to $\delta_0$ and
    hence $\widetilde{P}^n u$ converges to $u_0$ in $L_{1}({X},\lambda)$. In this case also
    the length of $\conv(\supp(\tilde\Lc^{*n}\mu))$ tends to $0$
  \item $\tilde\Lc^{*n}\mu\prec\delta_0$ for some $n\geq0$. Then
    $(\tilde\Lc^{*n}\mu)_n$ converges to $\mu_{-r_*}$ and hence $\tilde
    P^n u$ converges to $u_{-r_*}$ in $L_{1}({X},\lambda)$.
  \end{enumerate}
  {In the stable regime},
  only scenario (2) is possible, so that we always have convergence
  {of $(\tilde\Lc^{*n}\mu)_{n\geq0}$} to $\delta_0$.
\end{prop}

{Our arguments will rely on continuity and
monotonicity properties of the IFS, that we detail below, before proving
Proposition \ref{prop:summary} in Section \ref{se:propro}.}

\subsection{The IFS: continuity}

{\emph{Convergence in} $\mathsf{P}(Y)$, $\lim \mu_n = \mu$, will always mean weak convergence
of measures, $\int_Y \rho\;d\mu_n \to \int_Y \rho\;d\mu$ for bounded
continuous $\rho:Y\to \R$. Since $Y$ is a
bounded interval, this is equivalent to convergence in the}
\emph{Wasserstein-metric} $\d$ on $\mathsf{P}(Y)$.
If $F_\mu$ and $F_\nu$ are the distribution functions of $\mu$ and
$\nu$, then
\begin{equation}
  \label{eq:Wasserstein}
  \d(\mu,\nu):=\int_{-\infty}^\infty|F_\mu(x)-F_\nu(x)|\,dx .
\end{equation}
The Kantorovich-Rubinstein theorem (e.g. \cite[Ch.11]{Dudley02}) provides an additional
characterisation:
\begin{equation}
\label{eq:Kanto-Rubi}
  \d(\mu,\nu)=\sup_{\psi:\,\mathrm{Lip}_{Y}[
    \psi]\leq1}\int_{Y}\psi\,d(\mu-\nu)
\end{equation}
for any $\mu,\nu\in\mathsf{P}(Y)$.
Here, $\mathrm{Lip}_{Y}[\psi]  :=\sup_{y,y^{\prime}\in
Y;y\neq y^{\prime}}\left\vert \psi(y)-\psi(y^{\prime})\right\vert /\left\vert
y-y^{\prime}\right\vert $ for any $\psi:Y\rightarrow\mathbb{R}$ (and
analogously for functions on other domains).
{We now see that} there is a constant $K>0$
{(the common Lipschitz bound for the functions
$y \mapsto w_y(x)$ on $Y$, where $x\in X$)}
such that
\begin{align}
  \left\|\int_{{Y}} w_{\bullet}\,d\mu-\int_{{Y}} w_{\bullet}\,d\nu
  \right\|_{L_{1}({X},\lambda)}
  &\leq
  K\cdot\d(\mu,\nu).\label{eq:norm_W_estimate}
\end{align}
This means that, for densities from $\mathcal{D}'$, convergence of the
representing measures implies $L_1$-convergence of the densities.

We will also use the following estimate.
\begin{lemma}[{\textbf{Continuity of $ {(r,\mu) \mapsto} \Lc_r^*\mu$}}]
  \label{lemma:stability}
  There are constants $\kappa_1,\kappa_2>0$ such that
  \begin{displaymath}
    \d(\Lc_r^*\mu,\Lc_s^*\nu)\leq \kappa_1\d(\mu,\nu)+\kappa_2|r-s|
  \end{displaymath}
  for all $\mu,\nu\in\mathsf{P}(Y)$ and all $r,s\in R$.
\end{lemma}

\begin{proof}
  {}For Lipschitz functions $\psi$ on $Y$ we have
  \begin{gather}
    \label{eq:Lipschitz-1}
    \int_{Y}\psi\,d\left(\Lc_r^*\mu-\Lc_s^*\nu\right)
    \leq
    \mathrm{Lip}_{Y}\left[
      (\psi\circ\sigma_{r})\,p_{r}+(\psi\circ\tau_{r}
      )(1-p_{r})\right]  \cdot\d(\mu,\nu)\nonumber\\
    +\sup_{y\in Y}\,\mathrm{Lip}_{R}\left[  \psi(\sigma_{.}(y))\,p_{.}%
      (y)+\psi(\tau_{.}(y))(1-p_{.}(y))\right]  \cdot|r-s|
  \end{gather}
  Suppose that $\psi$ is $\mathcal{C}^{1}$, with $\left\vert \psi^{\prime
    }\right\vert \leq1$. Then the first Lipschitz constant is bounded by
  \begin{displaymath}
    \kappa_1:=\sup_{r \in R} {\left[
    \left\|(\tau_{r}-\sigma_{r})\cdot p_{r}^{\prime}\right\|_\infty
    +\left\|\sigma_{r}^{\prime}\cdot p_{r}+\tau_{r}^{\prime}\cdot
    (1-p_{r})\right\|_\infty
    \right]} < \infty,
\end{displaymath}
and the second one by
\begin{displaymath}
    \kappa_2:=\sup_{r \in R}   {\left[
    \left\| \frac{\partial\sigma_{r}}{\partial r}\cdot p_{r}\right\|_\infty
    +\left\| \frac{\partial\tau_{r}}{\partial r}\cdot(1-p_{r})\right\|_\infty
    +\left\| (\tau_{r}-\sigma_{r})\cdot\frac{\partial p_{r}}{\partial r}\right\|_\infty
    \right]} < \infty,
\end{displaymath}
and the lemma follows from the Kantorovich-Rubinstein theorem
\eqref{eq:Kanto-Rubi},
{since these $\mathcal{C}^1$ functions $\psi$
uniformly approximate the Lipschitz functions appearing there}.
\end{proof}

\begin{coro}
\label{co:contIFS}
{
The operators $\Lc^*_r$, $r\in R$, and $\tilde\Lc^*$ are uniformly}
Lipschitz-continuous on $\mathsf{P}(Y)$ for the Wasserstein metric.
\end{coro}

\begin{proof}
The Lipschitz-continuity of $\Lc^*_r$ is immediate from Lemma~\ref{lemma:stability}.
{}For $\tilde\Lc^*$, we recall from \eqref{eq:r_mu} and \eqref{Eq_P_tilda_acting_on_representing_measures}
that $\tilde\Lc^*\mu=\Lc^*_{r_\mu}\mu$ with $r_{\mu}=G(\int_{Y}\overline w\,d\mu)$ so that
$$|r_\mu-r_\nu|\leq \mathrm{Lip}(G)\; \mathrm{Lip}(\overline w)\; \d(\mu,\nu),$$
which allows to conclude with Lemma~\ref{lemma:stability}.
\end{proof}

\begin{remark}
  Rigorous numerical bounds give $\kappa_1\leq0.5761$ and
  $\kappa_2\leq0.5334$.
  These estimates can be used to show that not only the individual $\Lc^*_r$
  are uniformly contracting on $\mathsf{P}(Y)$, but also
  (under suitable restrictions on the function $G$)  $\tilde\Lc^*$ is a
  uniform contraction on $\mathsf{P}(Y^*)$ where $Y^*$ is a suitable
  neighbourhood of the support of $\mu_{r_*}$, and $\mu_{r_*}$ is the
  representing measure of $u_{r_*}$ with $r_*$ the unique positive fixed
  point of the equation $r=G(\phi(u_r))$. Our treatment of $\tilde\Lc^*$,
  however, does not rely on these estimates, because it is based on
  monotonicity properties explained below.
\end{remark}

\subsection{The IFS: monotonicity}
On the space $\mathsf{P}(Y)$ of {probability measures $\mu,\nu$ representing
densities from $\mathcal{D}',$} we introduce
an order {relation} by {defining}
\begin{equation}
  \label{eq:order}
  \mu\preceq\nu\quad:\Leftrightarrow\quad\forall y\in Y:\,
  \mu(y,\infty)\leq\nu(y,\infty)
\end{equation}
The symbols $\prec$, $\succeq$ and $\succ$ designate the usual variants of
$\preceq$.  We collect a few elementary facts on this order relation:
\begin{equation}
  \label{eq:mono-fact-1}
  \eqbox{t}{$\mu\preceq\nu$ if and only if
    $\textstyle{\int_Y} u\,d\mu\leq\int_Y u\,d\nu$ for
    each bounded and non-decreasing $u:Y\to\R$.}
\end{equation}
In particular, if $\mu\preceq\nu$, then
    ${\textstyle{\int_Y}} \overline{w}\,d\mu \leq
    {\textstyle{\int_Y}} \overline{w}\,d\nu$, and hence $r_{\mu} \leq r_{\nu}$ as well.
\begin{equation}
  \label{eq:mono-fact-2}
  \eqbox{t}{If $\mu\preceq\nu$ and if $\rho_1,\rho_2:Y\to Y$ are non-decreasing
    and such that
    $\rho_1(y)\leq\rho_2(y)$ for all $y$,
    then $\mu\circ\rho_1^{-1}\preceq\nu\circ\rho_2^{-1}$.}
\end{equation}
\begin{equation}
  \label{eq:mono-fact-2a}
  \eqbox{t}{If $\mu\preceq\nu$ and if
  ${\textstyle{\int_Y}} u\,d\mu
    = {\textstyle{\int_Y}} u\,d\nu$ for some
    strictly increasing\\ $u:Y\to\R$, then $\mu=\nu$.}
\end{equation}
\begin{equation}
  \label{eq:mono-fact-3}
  \eqbox{t}{Let $z\in Y$. Then $\delta_z\preceq\mu$ if and only if
    $\supp(\mu)\subseteq[z,\infty)$.}
\end{equation}
{We also observe that the representing measures $\mu_r$ of the
$T_r$-invariant densities $u_r$ form a linearly ordered subset of $\mathsf{P}(Y)$.
Routine calculations based on (\ref{eq:mu-r-density}) show that}
\begin{equation}
  \label{eq:mono-mu-r-family}
  {
  \mu_r \prec \mu_s \quad \text{if}\ r<s.
  }
\end{equation}

{Our analysis of the asymptotic behaviour of the IFS will crucially
depend on the fact that the operators $\Lc^*_r $ and $\tilde\Lc^* $ respect this order
relation on $\mathsf{P}(Y)$, as made precise in the next lemma:}

\begin{lemma}[{\textbf{Monotonicity of $(r,\mu) \mapsto \Lc_r^*\mu$}}]
  \label{lemma:pre-monotonicity}
  Let $\mu,\nu\in\mathsf{P}(Y)$ {and $r,s \in R$}.
  \begin{enumerate}[a)]
  \item\label{item:pre-monotonicity-a} If $\mu\prec\nu$,
    then $\Lc^*_r\mu\prec\Lc_r^*\nu$.
  \item\label{item:pre-monotonicity-b} If $r<s$, then
    $\Lc_r^*\mu\prec\Lc_s^*\mu$.
  \item\label{item:pre-monotonicity-c} If
    $\mu\prec\nu$, then
    $\tilde\Lc^*\mu\prec\tilde\Lc^*\nu$.
  \end{enumerate}
\end{lemma}

\begin{proof}
  Let $u:Y\to\R$ be bounded and non-decreasing and {recall} that
\begin{gather}
    \label{eq:explicit-1}
      \int_{{Y}} u\,d(\Lc_r^*\mu) =
      \int_{{Y}} {[} u\circ\sigma_r\cdot p_r+ u\circ\tau_r\cdot(1-p_r){]}\,d\mu.
\end{gather}

  \ref{item:pre-monotonicity-a}) Let $\mu\preceq\nu$.
  {In view of (\ref{eq:mono-fact-1}) we can prove $\Lc^*\mu\preceq\Lc^*\nu$ by showing
  that the integrand on the right-hand side of} \eqref{eq:explicit-1} is non-decreasing.
  {}For this we use the facts that
  $\sigma_r$ and $\tau_r$ are strictly increasing {with $\sigma_r < \tau_r$},
  and that $p_r$ is non-increasing, since
   $p_r'(y)=-\frac{4-r^2}{(4+ry)^2}<0$. One gets then, for $x {<} y$,
  \begin{equation}
    \label{eq:monotonicity-estimate}
    \begin{split}
      &u(\sigma_rx)p_r(x)+u(\tau_rx)(1-p_r(x))\\
      =&
      u(\sigma_rx)p_r(y)+u(\sigma_rx)\underbrace{(p_r(x)-p_r(y))}_{{>}0}
      +u(\tau_rx)(1-p_r(x))\\
      \leq&
      u(\sigma_rx)p_r(y)+u(\tau_rx)(p_r(x)-p_r(y))+u(\tau_rx)(1-p_r(x))\\
      =&
      u(\sigma_rx)p_r(y)+u(\tau_rx)(1-p_r(y))\\
      =&
      u(\sigma_ry)p_r(y)+u(\tau_ry)(1-p_r(y))\\
      &- \underbrace{\left[ \left( u(\sigma_ry)-u(\sigma_rx) \right) p_r(y)
                     + \left( u(\tau_ry)-u(\tau_rx) \right) (1-p_r(y)) \right].}_{\geq 0}
    \end{split}
  \end{equation}
  Hence $\mu\preceq\nu$ implies
  $\Lc^*_{{r}}\mu\preceq\Lc^*_{{r}}\nu$.
  {Now, if $u$ is strictly increasing, then
  (\ref{eq:monotonicity-estimate}) always is a strict inequality, i.e. the
  integrand on the right-hand side of \eqref{eq:explicit-1} is strictly increasing.
  Therefore, $\mu\prec\nu$ implies $\Lc^*_r\mu\prec\Lc^*_r\nu$ by (\ref{eq:mono-fact-2a}).} \\

  \ref{item:pre-monotonicity-b}) We must show that \eqref{eq:explicit-1} is
  non-decreasing as a function of $r$. To this end note first that
  \begin{gather}
    \frac{\partial p_r(y)}{\partial r}=\frac{y^2-4}{(ry+4)^2}<0\ ,\\
    \frac{\partial\sigma_r(y)}{\partial r}=\frac{8-2y^2}{(ry+y+r+4)^2}>0\ ,
    \label{eq:dp_dr}\\
    \frac{\partial\tau_r(y)}{\partial r}=\frac{8-2y^2}{(ry-y-r+4)^2}>0\ .
  \end{gather}
  Hence, if $r< s$, then
  \begin{displaymath}
    \begin{split}
      &u(\sigma_rx)p_r(x)+u(\tau_rx)(1-p_r(x))\\
      =&
      u(\sigma_rx)p_s(x)+u(\sigma_rx)\underbrace{(p_r(x)-p_s(x))}_{>0}
      +u(\tau_rx)(1-p_r(x))\\
      \leq&
      u(\sigma_sx)p_s(x)+u(\tau_sx)(p_r(x)-p_s(x))+u(\tau_sx)(1-p_r(x))\\
      =&
      u(\sigma_sx)p_s(x)+u(\tau_sx)(1-p_s(x))\ ,
    \end{split}
  \end{displaymath}
  and for strictly {increasing} $u$ we have indeed a strict inequality.\\

  \ref{item:pre-monotonicity-c}) This follows from
  \ref{item:pre-monotonicity-a}) and \ref{item:pre-monotonicity-b}):
  \begin{displaymath}
    \tilde\Lc^*\mu=\Lc_{r_\mu}^*\mu\prec\Lc_{r_\mu}^*\nu
    {\,\preceq\,} \Lc^*_{r_\nu}\nu
    =\tilde\Lc^*\nu
  \end{displaymath}
  as $r_\mu=G(\int_{{Y}}\overline{w}\,d\mu) {\,\leq\,}
  G(\int_{{Y}}\overline{w}\,d\nu)=r_\nu$ by
  \eqref{eq:mono-fact-1}.
\end{proof}

In \S\ref{subsec:stablemfd} below, we will also make use of a more precise
quantitative version of statement a).
It is natural to state and prove it at this point.

\begin{lemma}[{\textbf{Quantifying the growth of $\mu \mapsto \Lc_r^*\mu$}}]
 \label{lemma:gap-estimate}
 Suppose that $\alpha,\beta>0$ are such
 that $u'\geq\alpha$ and $\tau_r',\sigma_r'\geq\beta$.
 Then, for $\mu\preceq\nu$,
 \begin{equation}
   \label{eq:gap-estimate}
   \int_{{Y}} u\,d(\Lc_r^*\nu)-
   \int_{{Y}} u\,d(\Lc_r^*\mu)
   \geq
   \alpha\,\beta\,\left(
   \int_{{Y}}\id\,d\nu-\int_{{Y}}\id\,d\mu\right)\,.
 \end{equation}
\end{lemma}

\begin{proof}
  Observing that
  $u(\sigma_ry)-u(\sigma_rx)$ and $u(\tau_ry)-u(\tau_rx)$ are
  $\geq \alpha \beta (y-x)$, we find for the last expression in
  \eqref{eq:monotonicity-estimate} that
  $$ \left[ \left( u(\sigma_ry)-u(\sigma_rx) \right) p_r(y)
    + \left( u(\tau_ry)-u(\tau_rx) \right) (1-p_r(y)) \right] \geq \alpha
  \beta y - \alpha \beta x.$$ This turns \eqref{eq:monotonicity-estimate} into
  a chain of inequalities which shows that the function given by $v(x):=
  u(\sigma_rx)p_r(x)+u(\tau_rx)(1-p_r(x)) -\alpha \beta x$ is
  non-decreasing. Hence, by \eqref{eq:mono-fact-1}, $\mu\preceq\nu$ entails
  $\textstyle{\int_Y} v\,d\mu\leq\int_Y v\,d\nu$, which is
  \eqref{eq:gap-estimate}.
\end{proof}

\subsection{{Dynamics of the IFS and the
asymptotics of $\widetilde{P}$ on $\mathcal{D}'$}}
\label{se:propro}

{We are now going to clarify the asymptotic
behaviour of $\tilde\Lc^*$ on $\mathsf{P}(Y)$. In view of
(\ref{eq:norm_W_estimate}), this also determines the asymptotics of
$\widetilde{P}$ on $\mathcal{D}'$}{, and hence proves
Proposition \ref{prop:summary}.}

{Our argument depends on monotonicity properties
which we can exploit since the topology of weak convergence on
$\mathsf{P}(Y)$, conveniently given by the Wasserstein metric, is
consistent with the order relation introduced above. Indeed, one
easily checks:}
\begin{equation}
  \label{eq:mono-cty-fact-1}
  {
  \eqbox{t}{If $(\nu_n)$ and $(\bar{\nu}_n)$ are weakly convergent sequences in
  $\mathsf{P}(Y)$ with $\nu_n \preceq \bar{\nu}_n$ for all $n$, then
  $\lim \nu_n \preceq \lim \bar{\nu}_n$.}
  }
\end{equation}

Recall from \S ~\ref{subsec:infinite-sys} that, {in the bistable regime},
$r_*$ is the unique positive fixed point of the equation $r=G(\phi(u_r))$.
For convenience, we now let
$$r_*:=0 \quad \text{in the stable regime}.$$
Then, in either case, $u_r$ with representing measure $\mu_r$ is
fixed by $\widetilde{P}$ iff $r\in\{0,\pm r_*\}$. By (\ref{eq:mono-mu-r-family})
we have $\mu_{-r_*} \preceq \mu_{0}=\delta_0 \preceq \mu_{r_*}$ with strict
inequalities {in the bistable regime}.

\begin{lemma}[{\textbf{Convergence by monotonicity}}]
  \label{lemma:monotonicity}
  If $\mu\preceq\tilde\Lc^*\mu$, then
  $\mu\preceq\tilde\Lc^*\mu\preceq\tilde\Lc^{*2}\mu\preceq\dots$, and the
  sequence $(\tilde\Lc^{*n}\mu)_{n\geq0}$ converges weakly to a measure
  $\mu_{r}\succeq\mu$ with $r\in\{0,\pm r_*\}$. The same holds for $\succeq$
  instead of $\preceq$.
\end{lemma}

\begin{proof}
  The monotonicity of the sequence $(\tilde\Lc^{*n}\mu)_{n\geq0}$ follows immediately
  from Lemma~\ref{lemma:pre-monotonicity}\ref{item:pre-monotonicity-c}).
  {Because of (\ref{eq:mono-cty-fact-1}),
  it implies that the sequence can have at most one weak accumulation point.
  Compactness of $\mathsf{P}(Y)$ and continuity of $\tilde\Lc^*$ therefore ensure
  that $(\tilde\Lc^{*n}\mu)_{n\geq0}$ converges to a fixed point
  of $\tilde\Lc^*$, i.e. to one of the measures $\mu_r$ with $r\in\{0,\pm
  r_*\}$, and (\ref{eq:mono-cty-fact-1}) entails $\mu_r\succeq\mu$.}
  The proof for decreasing sequences is the same.
\end{proof}

The following lemma strengthens the previous one considerably.
{It provides uniform control, in terms of the
Wasserstein distance \eqref{eq:Wasserstein}, on the asymptotics
of large families of representing measures.}

\begin{lemma}[{\textbf{Convergence by comparison}}]
  \label{lemma:controlled-convergence}
  We have the following:
  \begin{enumerate}[a)]
  \item\label{item:controlled-convergence-a}
    {In the stable regime}, there exists a sequence $(\varepsilon_n)_{n\geq0}$ of
    positive real numbers converging to zero such that
    $$ \d(\tilde\Lc^{*n}\mu,\delta_0)\leq\varepsilon_n \quad\ \text{for}\
       \mu\in\mathsf{P}(Y)\ \text{and}\ n\in\N.$$
  \item\label{item:controlled-convergence-b} {In the bistable regime, for every
    ${y}>0$ there exists a sequence
    $(\varepsilon_n)_{n\geq0}$ of positive real numbers converging to zero such that}
    $$ \d(\tilde\Lc^{*n}\mu,\mu_{r_*})\leq\varepsilon_n \quad\ \text{for}\
       \mu\in\mathsf{P}(Y)\ \text{with}\  \mu\succeq\delta_{{y}}\
       \text{and}\ n\in\N.$$
    An analogous assertion holds for measures $\mu\preceq\delta_{{-y}}$.
\end{enumerate}
\end{lemma}

\begin{proof}
  {
  As $Y=[-\frac23,\frac23]$, we trivially have $\delta_{-2/3} \preceq \mu \preceq \delta_{2/3}$
  for all $\mu \in \mathsf{P}(Y)$.
  In particular, $\tilde\Lc^*\delta_{2/3}\preceq\delta_{2/3}$, and Lemma~\ref{lemma:monotonicity}
  ensures that $(\tilde\Lc^{*n}\delta_{2/3})_{n\geq0}$ converges.
  Due to Lemma~\ref{lemma:pre-monotonicity}\ref{item:pre-monotonicity-c}), we have
  $\delta_0 \preceq \mu_{r_*} = \tilde\Lc^{*n} \mu_{r_*} \preceq \tilde\Lc^{*n} \delta_{2/3}$
  for all $n\geq0$, showing, via (\ref{eq:mono-cty-fact-1}), that
  $\lim \tilde\Lc^{*n}\delta_{2/3} = \mu_{r_*}$.
  }
  In the same way one proves that $(\tilde\Lc^{*n}\delta_{-2/3})_{n\geq0}$
  converges to $\mu_{-r_*}$.
  {For the stable regime this means}
  that both sequences converge to $\delta_0=\mu_0$.\\

  \ref{item:controlled-convergence-a})
  Assume we are {in the stable regime}. By the above discussion,
  \begin{displaymath}
    \varepsilon_n:=\d(\tilde\Lc^{*n}\delta_{-2/3},\delta_0)+
    \d(\tilde\Lc^{*n}\delta_{2/3},\delta_0)
  \end{displaymath}
  tends to zero. For any $\mu$, (\ref{eq:mono-cty-fact-1}) guarantees
  $\tilde\Lc^{*n}\delta_{-2/3}\preceq\tilde\Lc^{*n}\mu\preceq\tilde\Lc^{*n}\delta_{2/3}$
  for all $n\geq0$. Hence $F_{\tilde\Lc^{*n}\delta_{-2/3}}(y)\geq
  F_{\tilde\Lc^{*n}\mu}(y)\geq F_{\tilde\Lc^{*n}\delta_{2/3}}(y)$ for all
  $y$, proving $\d(\tilde\Lc^{*n}\mu,\delta_0)\leq\varepsilon_n$.\\

  \ref{item:controlled-convergence-b}) Now consider {the bistable regime}.
  {Note first that if there is a suitable sequence $(\varepsilon_n)_{n\geq0}$
    for some ${y}>0$, then it also works for all ${y}'>{y}$. Therefore, there
    is no loss of generality if we assume that ${y}>0$ is so small that }
  \begin{equation}
    \label{eq:jump-from-y}
    \sigma_{G(\overline{w}({{y}}))}({y})
    = \left(\frac12+\frac{G'(0)}{12}\right){y}+\Oc({y}^2)
    > {y}\
  \end{equation}
  {(use \eqref{eq:field_w_y}, \eqref{eq:dp_dr} and \eqref{eq:sigma-tau-def} to see
  that this can be achieved.) }
  {Since $\Lc^*_r \delta_{{y}} = p_r({{y}})
  \delta_{\sigma_r({y})} +
  (1-p_r({{y}})) \delta_{\tau_r({{y}})}$ and
  $\sigma_r({{y}}) < \tau_r({{y}})$, we then have}
  $\delta_0\prec\delta_{{y}}\prec
  \Lc^*_{G(\overline{w}({{y}}))}\delta_{{y}}
  =\tilde\Lc^*\delta_{{y}}$,
  recall Lemma \ref{lemma:pre-monotonicity}. Lemma~\ref{lemma:monotonicity} then implies that
  $(\tilde\Lc^{*n}\delta_{{y}})_{n\geq0}$ converges to
  $\mu_{r_*}$. In view of the
  initial discussion, $(\tilde\Lc^{*n}\delta_{2/3})_{n\geq0}$ converges
  to $\mu_{r_*}$ as well, so that
  \begin{displaymath}
    \varepsilon_n:=\d(\tilde\Lc^{*n}\delta_{{y}},\mu_{r_*})+
    \d(\tilde\Lc^{*n}\delta_{2/3},\mu_{r_*})
  \end{displaymath}
  defines a sequence of reals converging to zero.
  {Now take any $\mu\in\mathsf{P}(Y)$ with $\mu\succeq\delta_{{y}}$,
  then $\tilde\Lc^{*n}\delta_{2/3}\succeq\tilde\Lc^{*n}\mu\succeq\tilde\Lc^{*n}\delta_{{y}}$
  for all $n\geq0$, and $\d(\tilde\Lc^{*n}\mu,\mu_{r_*})\leq\varepsilon_n$ follows
  as in the proof of \ref{item:controlled-convergence-a}) above.}
\end{proof}

{This observation enables us to determine the asymptotics
of $\tilde\Lc^{*n}\mu$ for any $\mu \in \mathsf{P}(Y)$ which is completely supported
on the positive half $(0,2/3]$ of $Y$ (meaning that $\mu\succ\delta_0$, cf.
(\ref{eq:mono-fact-3})), or on its negative half $[-2/3,0)$.}

\begin{coro}
  \label{coro:monotone-convergence}
  Let $\mu\in\mathsf{P}(Y)$.
  \begin{enumerate}[a)]
  \item\label{item:monotone-convergence-a}
    {In the stable regime}, the sequence $(\tilde\Lc^{*n}\mu)_{n\geq0}$ converges
    to $\delta_0$.
  \item\label{item:monotone-convergence-b} {In the bistable regime}, if
    $\mu\succ\delta_0$, then the sequence $(\tilde\Lc^{*n}\mu)_{n\geq0}$ converges
    to $\mu_{r_*}$. If $\mu\prec\delta_0$, it converges to
    $\mu_{-r_*}$.
  \end{enumerate}
\end{coro}

\begin{proof}
  \ref{item:monotone-convergence-a}) follows immediately from
  Lemma~\ref{lemma:controlled-convergence}\ref{item:controlled-convergence-a}).
  We turn to \ref{item:monotone-convergence-b}): Let $r:=\int_{{Y}}
  \overline{w}\,d\mu$. Then $r>0$ because
  {$\overline{w}>0$ on $(0,2/3]$ and}
  $\mu\succ\delta_0$. Therefore
  $\sigma_{r}(0)>  {\sigma_{0}(0)=} 0$.
  {Fix some ${y}$ as in
  Lemma~\ref{lemma:controlled-convergence}\ref{item:controlled-convergence-b}),
  w.l.o.g. ${y}\in(0,\sigma_r(0))$.}
  Then $\delta_0 \prec \delta_{{y}}
  \preceq \Lc_r^*\mu=\tilde\Lc^*\mu$
  {since $\sigma_{r}$ and $\tau_{r}$ map $\rm{supp}(\mu)$
  into $[\sigma_{r}(0),2/3]$,}
  so that indeed
  $\d(\tilde\Lc^{*n}\mu,\mu_{r_*})\to0$ as $n\to\infty$ by the lemma.
\end{proof}

It remains to investigate the convergence of sequences $(\tilde\Lc^{*n}\mu)_{n\geq0}$
when none of these measures can be compared (in the sense of $\prec$) to
$\delta_0$. To this end let
$[a_{0},b_{0}]:=Y$. Given a sequence of parameters $r_{1},r_{2},\ldots\in R$
define
\[
a_{n}:=\sigma_{r_{n}}\circ\ldots\circ\sigma_{r_{1}}(a_0)\quad\text{and}\quad
b_{n}%
:=\tau_{r_{n}}\circ\ldots\circ\tau_{r_{1}}(b_0)
\]
for $n\geq1$, and, for any $\mu=\mu_0\in\mathsf{P}(Y)=\mathsf{P}[a_0,b_0]$,
consider the measures
\[
\mu_{n}:=\mathcal{L}_{r_{n}}^{\ast}\circ\ldots\circ\mathcal{L}_{r_{1}}^{\ast
}\mu\text{.}%
\]
Then \textrm{supp}$(\mu_{n}) \subseteq$
{ \textrm{supp}$(\mathcal{L}_{r_{n}}^{\ast}  \mu_{n-1}) \subseteq
\sigma_{r_{n}}([ a_{n-1},b_{n-1}]) \cup \tau_{r_{n}}([ a_{n-1},b_{n-1}]) \subseteq$}
$[ a_{n},b_{n}]$ by induction. Write
$[a,b]_{\varepsilon}:=[a-\varepsilon,b+\varepsilon]$,
{where $\varepsilon \geq 0$}. The next lemma exploits
the crucial observation that the two branches $\sigma_r$ and $\tau_r$ have
tangential contact at their common zero $z_r$, see  Figure~\ref{fig:sigma-tau}.

\begin{lemma}[{\textbf{Support intervals close to zeroes}}]
\label{lemma:crucial}
  {There exists some $C \in (0,\infty)$ such that the following holds:}
  Suppose that $(r_{n})_{n\geq1}$ is any given sequence in $R$. If for some
  $\varepsilon\geq0$ {and $\bar{n}(\varepsilon) \geq 0$} we have
  \begin{equation}
    z_{r_{n+1}}\in\lbrack a_{n},b_{n}]_{\varepsilon}\quad \text{for}\ n\geq
    \bar{n}(\varepsilon)\text{,}\tag{${\clubsuit}_\varepsilon$}%
  \end{equation}
  then
  \begin{equation}
    \label{eq:club-1}
    \overline{\lim_{n\rightarrow\infty}}\left\vert b_{n}-a_{n}\right\vert \leq
    C\varepsilon^{2}\text{ ,}
  \end{equation}
  \begin{equation}
    \label{eq:club-2}
    \overline{\lim_{n\rightarrow\infty}}\max\left(
      \left\vert a_{n}\right\vert ,\left\vert b_{n}\right\vert \right)
    \leq
      {\frac{3}{4}\varepsilon+C\varepsilon^{2}}  \text{,}%
  \end{equation}
  and, in case $\varepsilon=0$,
  \begin{equation}
    \label{eq:club-3}
    0\in [a_{n},b_{n}]\quad \text{for}\ n\geq \bar{n}(0)+1 \text{.}
  \end{equation}
\end{lemma}

\begin{proof}
  Let $\varepsilon\geq0$ and assume
  ($\clubsuit_\varepsilon$). Note that, {for $n \geq \bar{n}=\bar{n}(\varepsilon)$},
  \begin{align*}
    &\text{if }a_{n} >z_{r_{n+1}}\text{, then }0<a_{n+1}<3/4\cdot\varepsilon,\\
    &\text{if }b_{n} <z_{r_{n+1}}\text{, then }-3/4\cdot\varepsilon<b_{n+1}<0,\\
    &\text{if }a_{n} \leq z_{r_{n+1}}\leq b_{n}\text{, then }a_{n+1}\leq0\leq
    b_{n+1}.
  \end{align*}
  {The first implication holds because
  $0 = \sigma_{r_{n+1}}(z_{r_{n+1}}) < \sigma_{r_{n+1}}(a_{n}) = a_{n+1}$ as
  $\sigma_{r_{n+1}}$ increases strictly, and since by ($\clubsuit_\varepsilon$)
  we have $a_{n} \in (z_{r_{n+1}},z_{r_{n+1}}+\varepsilon]$, whence
  $a_{n+1} < \varepsilon \cdot \sup \sigma_{r_{n+1}}' \leq 3\varepsilon/4$
  due to (\ref{eq:less3/4}). Analogously for the second implication.
  The third is immediate from monotonicity. }

  Now, as $\sigma_{r}$ and $\tau_{r}$ share a common zero $z_r$, (\ref{eq:less3/4})
  ensures $ b_{{\bar{n}+m}}-a_{{\bar{n}+m}}
      \leq \frac{3}{4}(b_{{\bar{n}+m}-1}-a_{{\bar{n}+m}-1}) $
  in case $z_r \in [a_{{\bar{n}+m}-1},b_{{\bar{n}+m}-1}]$.
  Otherwise, note that $z_r$ is $\varepsilon$-close to one of the endpoints, w.l.o.g.
  to $a_{{\bar{n}+m}-1}$. Since $\sigma_{r}$ and $\tau_{r}$ are tangent
  at $z_{r}$, there is some $C>0$ s.t.
  $0 \leq \tau_{r_{\bar{n}+m}}(a_{{\bar{n}+m}-1})-\sigma_{r_{\bar{n}+m}}(a_{{\bar{n}+m}-1})
  \leq \frac{C}{4} {\varepsilon}^2$ in this case, while (\ref{eq:less3/4})
  controls the rest of $ b_{{\bar{n}+m}}-a_{{\bar{n}+m}}$.
  In view of $\rm{diam}(Y)=4/3$, we thus obtain, for $m \geq 1$,
  $$ b_{{\bar{n}+m}}-a_{{\bar{n}+m}}
  \leq \frac{3}{4}(b_{{\bar{n}+m}-1}-a_{{\bar{n}+m}-1})
  +\frac{C}{4}\varepsilon^{2}
  \leq \dots \leq \frac{4}{3} \left(\frac{3}{4}\right)^{{m}}+C \varepsilon^{2}. $$
  Statement (\ref{eq:club-1}) follows immediately.
  {}For the asymptotic estimate (\ref{eq:club-2}) on
  $\max\left(\left\vert a_{n}\right\vert ,\left\vert b_{n}\right\vert \right)
   = \max(-a_n,b_n)$, use the above inequality plus the observation that,
  by the first two implications stated in this proof,
  $a_{\bar{n}+m}$ and $-b_{\bar{n}+m}$ never exceed $3 \varepsilon /4$.
  {}Finally, if $\varepsilon=0$, (\ref{eq:club-3}) is straightforward from
  ($\clubsuit_\varepsilon$) and the third implication above.
\end{proof}

{While the full strength of this lemma will only be required
in the next subsection, the $\varepsilon=0$ case enables us to now conclude the}

\begin{proof}[{Proof of Proposition~\ref{prop:summary}.}]
  {The conclusions of} (1) and (3) follow from
  Corollary~\ref{coro:monotone-convergence}. If
  neither of these two cases applies, then the assumption of (2) must be
  satisfied, and so condition ($\clubsuit_0$) of Lemma~\ref{lemma:crucial} is
  satisfied with $\bar{n}(0)=0$. Hence $\lim_{n\to\infty}\max(|a_n|,|b_n|)=0$
  by (\ref{eq:club-2}). As the $\tilde\Lc^{*n}\mu$ are
  supported in $[a_n,b_n]$, these measures must converge to $\delta_0$.
\end{proof}

\section{Proofs: {the self-consistent PFO for} the infinite{-size} system}
\label{sec:infinite-sys-proofs}

\subsection{{Shadowing densities and the asymptotics of $\widetilde{P}$ on $\mathcal{D}$}}
We are now going to clarify the asymptotics of the self-consistent PFO on the
set $\mathcal{D}$ of all densities, proving

\begin{prop} [ {\textbf{Long-term behaviour of }$\widetilde{P}$\textbf{\ on
  }$\mathcal{D}$} ]
  \label{prop:summary2}
  {For every $u \in \mathcal{D}$, the sequence $(\widetilde{P}^n u)_{n\geq0}$
  converges in $L_1(X,\lambda)$, and}
  $$
  \lim_{n\to\infty} \widetilde{P}^n u
  \quad
  \begin{cases}
    \,\,=\,u_0& \text{in the stable regime,}\\
    \,\,\in \{u_{-r_*},u_0,u_{r_*}\}&\text{in the bistable regime.}
  \end{cases}
  $$
  The basins $\{u\in\mathcal{D}: \lim_{n\to\infty} \widetilde{P}^n u = u_{\pm
    r_*}\}$ of the stable fixed points $u_{\pm r_*}$ are $L_1$-open.
\end{prop}

(The set of densities attracted to $u_0$ in the bistable regime
will be discussed in \S\ref{subsec:stablemfd} below.)\\

{We begin with some notational preparations. Throughout, we fix some}
$u\in\mathcal{D}$. The iterates $\widetilde{P}^n u$ define parameters
$r_n:=G(\phi(\widetilde{P}^{n-1}u))$ ($n\geq1$). With this notation, $\tilde
P^nu=P_{r_{n}} \dots P_{r_1}u$.

We let $\pi_{N}$, $N\geq1$, denote the partition of $X$ into
monotonicity intervals of $T_{r_{N}}\circ\ldots\circ T_{r_{1}}$.
{Note that each branch of this map is a fractional
linear bijection from a member of $\pi_{N}$ onto $X$. Since the $T_r$,
$r\in R$, have a common uniform expansion rate, we see that }
$\mathrm{diam}(\pi_{N})\rightarrow0$, and
{hence, by the standard martingale convergence theorem, }
$\mathbb{E}[u\parallel\sigma(\pi_{N})]\rightarrow u$ in $L_{1}({X},\lambda)$, that
is,
\begin{equation}
  \label{eq:eta-n-cge-0}
  \eta_{N}:=\left\| \mathbb{E}[u\parallel\sigma(\pi_{N})]-u\right\|
  _{L_{1}({X},\lambda)}\longrightarrow0\quad \text{ as }N\rightarrow\infty\text{.}%
\end{equation}
Write
\[
v_{{k}}^{(N)}:=P_{r_{N+{k}}} \ldots  P_{r_{1}}
\left(  \mathbb{E}[u\parallel\sigma(\pi_{N})]\right)\quad  \text{for}\
{k}\geq0\ \text{and}\ N\geq1 ,
\]
and observe that $v_{{k}}^{(N)}\in\Dc'$ because it is a weighted sum of images of
the constant function $1$ under various fractional linear branches
{(recall (\ref{eq:transfer-by-fract-lin}) and (\ref{Eq_Invariance_of_Y}))}.

For $N=0$ we
let $v_0^{(0)}:=u_{r_*}$
and write, in analogy to the notation introduced
for $N\geq1$, $v_k^{(0)}:=P_{r_{{k}}} \ldots  P_{r_{1}}(v_0^{(0)})$ and $\eta_0:=\|v_0^{(0)}-u\|_{L_1(X,\lambda)}$.
Obviously, $v_k^{(0)}\in\Dc'$ for all $k\geq0$.

Hence there are measures $\mu_{{k}}^{(N)}\in\mathsf{P}(Y)$ such that
$v_{{k}}^{(N)}=\int_{{Y}} w_{\bullet}\,d\mu_{{k}}^{(N)}$.
Observe also that
\begin{equation}
  \label{eq:shadow-estimate}
  \|\widetilde{P}^{N+{k}}u-v_{{k}}^{(N)}\|_{L_{1}({X},\lambda)}
  \leq\eta_N\quad\text{for all ${k}\geq0$ and $N\geq0$,}
\end{equation}
{as $\| P_r \|=1$ for all $r$}, so that in particular
\begin{equation}
   \label{eq:no-good-name}
   |\phi(v_{{k}}^{(N)})-\phi(\widetilde{P}^{N+{k}}u)|
   \leq  {\eta_{N}}, \,
   | G(\phi(v_{{k}}^{(N)}))-r_{N+{k}+1}|
   \leq \left\Vert G^{\prime}\right\Vert _{\infty}\cdot{{\eta_{N}}}
\end{equation}
{In addition, we need to understand the distances }
$$ \Delta_n^{(N,k)} :=
    \|v_{n+k}^{(N)}-\widetilde{P}^nv_k^{(N)}\|_{L_{1}({X},\lambda)} . $$
{which, in fact, admit some control which is uniform in $k$:}

\begin{lemma}[{\textbf{Shadowing control}}]
\label{lemma:shadow-estimate}
{There is a non-decreasing sequence $(\Delta_n)_{n\geq 0}$ in $(0,\infty)$,
not depending on $u\in \Dc$,
such that }
\begin{equation}
  \label{eq:shadow-estimate-new}
  {\Delta_n^{(N,k)} \leq \eta_N \cdot \Delta_n\ \quad
  \text{for}\ k,n\geq0\ \text{and}\
    N\geq0
}
\end{equation}
\end{lemma}

\begin{proof}
{Let} $r_n^{(N,k)}:=G(\phi(\widetilde{P}^{n-1}v_k^{(N)}))$,
{and observe that \eqref{eq:norm_W_estimate} entails}
\begin{displaymath}
  \begin{split}
    \Delta_n^{(N,k)}
    =&
    \|P_{r_{N+n+k}}\dots P_{r_{N+1+k}}v_k^{(N)}-P_{r_{n}^{(N,k)}}\dots
    P_{r_{1}^{(N,k)}}v_k^{(N)}\|_{L_{1}({X},\lambda)}\\
    \leq&
    K\cdot\d(\Lc^*_{r_{N+n+k}}\dots \Lc^*_{r_{N+1+k}}\mu_k^{(N)}
    , \Lc^*_{r_{n}^{(N,k)}}\dots\Lc^*_{r_{1}^{(N,k)}}\mu_k^{(N)}).
  \end{split}
\end{displaymath}
Applying Lemma~\ref{lemma:stability} repeatedly, we therefore see that
\begin{displaymath}
  \begin{split}
    \Delta_n^{(N,k)}
    &\leq
    K\,\kappa_2\sum_{i=0}^{n-1}\kappa_1^{i}\,|r_{N+n+k-i}-r_{n-i}^{(N,k)}|\\
    &=
    K\,\kappa_2\sum_{i=0}^{n-1}\kappa_1^{i}\,|G(\phi(\tilde
    P^{N+n+k-i-1}u))-G(\phi(\widetilde{P}^{n-i-1}v_k^{(N)}))|\\
    &\leq
    K\,\|G'\|_\infty\,\kappa_2\sum_{i=0}^{n-1}\kappa_1^{i}\,
    \|\tilde  P^{N+n+k-i-1}u-\widetilde{P}^{n-i-1}v_k^{(N)}\|_{L_{1}({X},\lambda)}\\
    &\leq
    K\,\|G'\|_\infty\,\kappa_2\sum_{i=0}^{n-1}\kappa_1^{i}
    \left(\eta_N+\Delta_{n-i-1}^{(N,k)}\right) ,
  \end{split}
\end{displaymath}
where the last inequality uses \eqref{eq:shadow-estimate}.
{(Recall that $\widetilde{P}$ does not contract on $L_{1}(X,\lambda)$,
whence the need for the $\Delta_{n-i-1}^{(N,k)}$-term.)
Letting $K_n:=1+K\,\|G'\|_\infty\,\kappa_2\sum_{i=0}^{n-1}\kappa_1^{i}$ and }
$\widehat{\Delta}^{(N,k)}_n:=\max\{\Delta_i^{(N,k)}:i=0,\dots,n-1\}$,
{we thus obtain }
\begin{equation}
  \label{eq:punchline-1}
  \widehat{\Delta}^{(N,k)}_n
  \leq
  K_n\cdot(\eta_N+\widehat{\Delta}^{(N,k)}_{n-1})
  \leq\dots\leq
  \eta_N\cdot n\,K_n^n ,
\end{equation}
{which proves our assertion.}
\end{proof}

{We can now complete the }

\begin{proof}[{Proof of Proposition \ref{prop:summary2}.}]
We begin with the easiest situation:\\

\paragraph{\textbf{{The stable regime}}}
We have to show that
$\lim_{n\to\infty}\|\widetilde{P}^nu-u_0\|_{L_{1}({X},\lambda)}=0$.
{Take any $\varepsilon >0$. Let $(\varepsilon_n)_{n\geq0}$ be the
sequence provided by Lemma~\ref{lemma:controlled-convergence}\ref{item:controlled-convergence-a}),
and $K$ the constant from (\ref{eq:norm_W_estimate}). There is some $n$ (henceforth fixed) for
which $K \varepsilon_n < \varepsilon/3$. In view of (\ref{eq:eta-n-cge-0}), there is some $N_0$
such that $(1+\Delta_n) \eta_N < 2\varepsilon/3$ whenever $N \geq N_0$. We then find, using
(\ref{eq:shadow-estimate}), Lemma \ref{lemma:shadow-estimate}, and (\ref{eq:norm_W_estimate}) together with
Lemma~\ref{lemma:controlled-convergence}\ref{item:controlled-convergence-a}) that }
\begin{equation}
\label{eq:long-estimate-1}
\begin{split}
  \|\widetilde{P}^{N+n}u-u_0\|_{L_{1}({X},\lambda)}
  &\leq
  \eta_N+\|v_n^{(N)}-\widetilde{P}^nv_0^{(N)}\|_{L_{1}({X},\lambda)}
  +{K} \varepsilon_n\\
  &\leq
  \eta_N+{\Delta_n} \eta_N + {K} \varepsilon_n\\
  &<
  \varepsilon\ \quad \text{for}\ N\geq N_0,
\end{split}
\end{equation}
which completes the proof in this case.\\

\paragraph{\textbf{{The bistable regime}}}

Given the sequence $r_n=G(\phi(\widetilde{P}^{n-1}u))$ as before, we let
$[a_n,b_n]\subseteq Y$ be the sequence of parameter intervals from
Lemma~\ref{lemma:crucial}. Observe that the measures representing the
$v_n^{(N)}$ satisfy \textrm{supp}$({\mu_{n}^{(N)}})\subseteq [a_{n},b_{n}]$
for all $n$ and $N$.
We now distinguish two cases:\\

\emph{First case:} For all $\varepsilon>0$ we have ($\clubsuit_{\varepsilon}$)
from Lemma~\ref{lemma:crucial}.  {Then, for any $\varepsilon>0$, the lemma
  ensures that there is some $n$ (henceforth fixed) with $\max\left(
    \left\vert a_{n}\right\vert ,\left\vert b_{n}\right\vert \right) <
  \varepsilon / 4K$, so that also $\d(\mu_n^{(N)},\delta_0) < \varepsilon
  /2K$, whatever $N$. Due to (\ref{eq:eta-n-cge-0}), $\eta_N < \varepsilon /2$
  for $N \geq N_0$, and we find, using (\ref{eq:shadow-estimate}) and
  (\ref{eq:norm_W_estimate}),}
\begin{displaymath}
\begin{split}
  \|\widetilde{P}^{N+n}u-u_0\|_{L_{1}({X},\lambda)}
  &\leq
  \eta_N+\|v_n^{(N)}-u_0\|_{L_{1}({X},\lambda)}\\
  &\leq
  \eta_N+ K \d(\mu_n^{(N)},\delta_0)\\
  &<
  \varepsilon\ \quad \text{for}\ N\geq N_0,
\end{split}
\end{displaymath}
{showing that indeed $\widetilde{P}^{n}u \to u_0 $.}\\

\emph{Second case:} there is some $\overline{\varepsilon}>0$ s.t.
($\clubsuit_{\overline{\varepsilon}}$) is violated in that, say,
\begin{equation}
  \label{eq:exceptional:n}
  z_{r_{n}}<a_{n-1}-\overline{\varepsilon}
\end{equation}
for infinitely many $n$.
{We show that this implies
$\widetilde{P}^{n}u \to u_{r_*} $.
(If ($\clubsuit_{\overline{\varepsilon}}$) is violated
in the other direction, $\widetilde{P}^{n}u \to u_{-r_*} $
then follows by symmetry.)}\\

{In view of (\ref{eq:no-good-name}), and since (due to
$\mu_k^{(N)} \succeq \delta_{a_{N+k}}$, (\ref{eq:mono-fact-1}), and
(\ref{Eq_mean_for_convex_combi}))
$\phi(v_k^{(N)}) \geq \phi(w_{a_{N+k}}) = \bar{w}(a_{N+k})$, we have  }
\begin{displaymath}
\begin{split}
r_{N+k+1} &=      G(\phi(\widetilde{P}^{N+k}u)) \geq G(\phi(v_k^{(N)})-\eta_N)\\
          &\geq   {G(\phi(v_k^{(N)}))-\left\Vert G^{\prime}\right\Vert _{\infty} \eta_N}
          \geq G(\overline{w}(a_{N+k})) - {\left\Vert G^{\prime}\right\Vert _{\infty}} \eta_N ,
\end{split}
\end{displaymath}
and hence, observing that $\|\frac{\partial\sigma_r}{\partial
  r}\|_\infty\leq1$ and writing $\tilde\sigma(y) :=
\sigma_{G(\overline{w}(y))}(y)$ for $y \in Y$,
\begin{equation}
  \label{eq:sigma-tilde}
  a_{N+k+1} =\sigma_{r_{N+k+1}}(a_{N+k})
  \geq \tilde\sigma(a_{N+k})
  - {\left\Vert G^{\prime}\right\Vert _{\infty}} \eta_{N}
\end{equation}
for all $N$ and $k$. Note that $\tilde\sigma'(0)=\sigma_0'(0)+
\frac{\partial}{\partial r} \sigma_r(0)|_{r=0} \cdot
G'(0)\cdot\overline{w}'(0)=\frac12+\frac12\cdot G'(0)\cdot\frac16>1$, see
\eqref{eq:dp_dr} and \eqref{eq:field_w_y}. {Therefore, if we} fix some
$\omega\in(1,\tilde\sigma'(0))$, there exists some $a^*>0$ such that
$\tilde\sigma(a)\geq\omega a$ for all
$a\in(0,a^*{]}$. Without loss of generality, $\overline\varepsilon/3<a^*$.\\

Now fix $N$ such that ${\left\Vert G^{\prime}\right\Vert _{\infty}}
\eta_N<(\omega-1)\,\overline\varepsilon/3$, and let $N+n+1$ {satisfy
  (\ref{eq:exceptional:n})}.  Due to \eqref{eq:z_r}, we have
\begin{equation}
   \label{eq:a-lower}
   a_{N+n+1}=\sigma_{r_{N+n+1}}(a_{N+n}) >
   \sigma_{r_{N+n+1}}(z_{r_{N+n+1}}+\overline\varepsilon)>\overline\varepsilon/3 .
\end{equation}
Now, if $a_{N+n+1}\geq a^*$, then, by (\ref{eq:sigma-tilde}),
\begin{displaymath}
  \begin{split}
   a_{N+n+2} &\geq \tilde\sigma(a_{N+n+1}) - {\left\Vert G^{\prime}\right\Vert _{\infty}} \eta_N\\
             &\geq \tilde\sigma(a^*) - (\omega-1)\overline\varepsilon/3\\
             &\geq \omega a^*-(\omega-1)a^*=a^*>\overline\varepsilon/3.
  \end{split}
\end{displaymath}
Otherwise, $a_{N+n+1} \in (0,a^*)$, and again
\begin{displaymath}
  \begin{split}
   a_{N+n+2} &\geq \tilde\sigma(a_{N+n+1}) - {\left\Vert G^{\prime}\right\Vert _{\infty}}  \eta_N\\
             &> \omega\overline\varepsilon/3-(\omega-1)\,\overline\varepsilon/3
                = \overline\varepsilon/3.
  \end{split}
\end{displaymath}
It
follows inductively that $\liminf_ka_{k}\geq\overline\varepsilon/3$.
More precisely: If $N_1$ and $n_1$ are integers such that 
  $\eta_{N_1} < \varrho:={\left\Vert G^{\prime}\right\Vert _{\infty}}^{-1}
  (\omega-1)\,\overline\varepsilon/3$
and
$a_{N_1+n_1}>z_{r_{N_1+n_1+1}}+\overline\varepsilon$, then
\begin{displaymath}
  a_k > \overline\varepsilon /3 \quad \text{for } k>N_1+n_1.
\end{displaymath}
  In particular, if the initial density $u$ is such that
  $\eta_0=\|u_{r_*}-u\|_{L_1(X,\lambda)}<\varrho$, then
  we can take $N_1=0$.
\\

Next,
fix ${y}:=\overline\varepsilon/6\in(0,\overline\varepsilon/3)$,
and
choose a sequence $(\varepsilon_n)_{n\geq0}$ according to
Lemma~\ref{lemma:controlled-convergence}\ref{item:controlled-convergence-b}).
Then $0 < {y} < a_k$ and hence
$\delta_0\prec\delta_{{y}}\preceq\mu_k^{(N)}$ for
$k>N_1+n_1$ so that the lemma
implies $\d(\tilde\Lc^{*n}\mu_k^{(N)},\mu_{r_*})\leq\varepsilon_n$.
Hence, {by (\ref{eq:norm_W_estimate}),}
\begin{displaymath}
  \|\tilde P^nv_k^{(N)}-u_{r_*}\|_{L_{1}({X},\lambda)}
  \leq K\cdot\varepsilon_n\quad\text{for
     $k>N_1+n_1$ and all $n,{N}$.}
\end{displaymath}
We then find, {using (\ref{eq:shadow-estimate}) and
  Lemma~\ref{lemma:shadow-estimate},}
\begin{equation}
  \label{eq:N+k+n-estimate}
  \begin{split}
    \|\widetilde{P}^{N+k+n}u-u_{r_*}\|_{L_{1}({X},\lambda)} &\leq
    \eta_N+\|v_{k+n}^{(N)}-\tilde
      P^nv_k^{(N)}\|_{L_{1}({X},\lambda)}+K\cdot\varepsilon_n\\
    &\leq
    \eta_N+{\Delta_n \eta_N}+K\cdot\varepsilon_n
  \end{split}
\end{equation}
{for $k>N_1+n_1$ and all $n,N$.}
Now $\lim_{n\to\infty}\|\widetilde{P}^nu-u_{r_*}\|_{L_{1}({X},\lambda)}=0$
{follows as in the} {stable} case.\\

It remains to prove that the basin of attraction of $u_{r_*}$ is $L_1$-open.
{(Then, by symmetry, the same is true for $u_{-r_*}$.)}
As $\widetilde P$ is $L_1$-continuous, it suffices to show that this basin
contains an open $L_1$-ball centered at $u_{r_*}$. To check the latter
condition, first notice that $z_{r_*}<0<\supp(\mu_{r_*})$ so that there is
some $n_1>0$ such that $\sigma_{r_*}^{n_1}(a_0)>0$. As we can assume
w.l.o.g. that $\overline\varepsilon<|z_{r_*}|$, we have
$\sigma_{r_*}^{n_1}(a_0)>z_{r_*}+\overline\varepsilon$, and as $\widetilde P$
is $L_1$-continuous, there is some $\overline\varrho\in(0,\varrho)$ such that
$a_{r_{n_1}}=\sigma_{r_{n_1}}\circ\dots\circ\sigma_{r_1}(a_0)>z_{r_{n_1}}+\overline\varepsilon$
whenever $\|u-u_{r_*}\|_{L_1(X,\lambda)}<\overline\varrho$.  Therefore we can
continue to argue as in the previous paragraph (using the present $n_1$ and
$N_1=0$) to conclude that $\lim_{n\to\infty}\|\widetilde
P^nu-u_{r_*}\|_{L_1(X,\lambda)}=0$.

\end{proof}
 
\begin{remark}
  We just proved a bit more than what is claimed in
  Proposition~\ref{prop:summary2}: another look at equation
  \eqref{eq:N+k+n-estimate} reveals that, in the bistable regime, the stable
  fixed point $u_{r_*}$ of $\widetilde P$ is even Lyapunov-stable (and the
  same is true for $u_{-r_*}$). Indeed, fix
  $\overline\varepsilon>0$, $n_1\in\N$ and $\overline\varrho>0$ as in the preceding paragraph. That
  choice was completely independent of the particular initial densities
  investigated there, and the same is true of the choice of the constants $K$,
  $\Delta_n$ and $\varepsilon_n$ occuring in estimate
  \eqref{eq:N+k+n-estimate}.  Now let $\delta>0$. Choose $n_2\in\N$ such that
  $\varepsilon_{n_2}<\frac\delta{2K}$ and then
  $\eta:=\min\{\overline\varrho,\frac\delta{2(1+\Delta_{n_2})}\}$. Then equation
  \eqref{eq:N+k+n-estimate}, applied with $N=0$, shows that for each $u\in
  L_1(X,\lambda)$ with $\eta_0=\|u-u_{r_*}\|_{L_1(X,\lambda)}<\eta$ and for
  each $n\geq0$ holds
  \begin{displaymath}
    \|\widetilde P^{n_1+n_2+n}u-u_{r_*}\|_{L_1(X,\lambda)}
    \leq
    \eta_0(1+\Delta_{n_2})+K\,\varepsilon_{n_2}
    <
    \delta\ .
  \end{displaymath}

\end{remark}

\subsection{The stable manifold of $u_0$ {in the bistable regime}}
\label{subsec:stablemfd}
Let $W^s(u_0):=\{u\in\Dc:\widetilde{P}^n u\to u_0\}$ denote the \emph{stable manifold} of
$u_0$ in the space of \emph{all} probability densities on $X$. Clearly, all
symmetric densities $u$ (i.e. those satisfying $u(-x)=u(x)$) belong to $W^s(u_0)$,
because symmetric densities have field $\phi(u)=0$ so that also the parameter
$G(\phi(u))=0$, and symmetry is preserved under the operator $P_0$.\\


However, $W^s(u_0)$ is not a big set. In the present section we prove

\begin{prop}[{\textbf{The basins of $u_{\pm r_*}$ touch $W^s(u_0)\cap\Dc'$}}]
  \label{prop:basin-boundary}
  Each density in $W^s(u_0)\cap\Dc'$ belongs to the boundaries of the basins
  of $u_{r_*}$ and of $u_{-r_*}$.
\end{prop}

We start by providing more information on the
fields $\phi(\widetilde{P}^n u)$ of orbits in $W^s(u_0)\cap\Dc'$.
Recall that for $u=\int_Y w_{\bullet}\,d\mu\in\Dc'$ we have $\widetilde{P}^n u=\int_Y
w_{\bullet}\,d(\tilde\Lc^{*n}\mu)$ $(n\geq0)$. Given such a density, we denote by $R_n(u)$
the ``radius'' of the support of $\tilde\Lc^{*n}\mu$,
i.e. $R_n(u):=\inf\{\varepsilon>0:\supp(\tilde\Lc^{*n}\mu)\subseteq[-\varepsilon,\varepsilon]\}$,
and let $\phi_n(u):=\phi(\widetilde{P}^n u)=\int_Y\overline{w}\,d(\tilde\Lc^{*n}\mu)$.

\begin{lemma}[{\textbf{Field versus support radius}}]
  \label{lemma:roland}
  In the bistable regime,
  for each $u\in W^s(u_0)\cap\Dc'$ there exists a constant $C_u>0$ such that
  \begin{equation}
  |\phi_n(u)|\leq C_u\cdot(R_n(u))^2\quad \text{for }\ n\geq0.
  \end{equation}
\end{lemma}

\begin{proof}
  In view of the explicit formula \eqref{eq:field_w_y}, we have
  $\overline{w}'(0)=\frac16$ and $\overline{w}''(0)=0$, and therefore see that
  there is some $\overline\varepsilon\in(0,\frac13)$ such that for every
  $\varepsilon\in(0,\overline\varepsilon)$ and all
  $y\in[-2\varepsilon,2\varepsilon]$,
  \begin{equation}
    \label{Eq_estimate_F_near_zero}%
    \begin{split}
      \frac{\left\vert y\right\vert }{6}\leq\left\vert \overline{w}(y)\right\vert
      \leq\frac{\left\vert y\right\vert }{6-6\varepsilon^2}\quad\text{ and }\quad\left\vert
        G(y)\right\vert >(B-c\varepsilon)|y|\text{,}
    \end{split}
  \end{equation}
  where $B:=G'(0)>6$ and $c$, too, is a positive constant which only depends
  on the function $G$. In addition, elementary calculations based on
  \eqref{eq:sigma-tau-def} and \eqref{eq:sigma_prime} show that letting
  $\kappa:=\max(1,\frac{B+2}{6})$, $\overline\varepsilon$ can be chosen such
  that, for every $\varepsilon\in(0,\overline\varepsilon)$ and
  $r\in[0,B\varepsilon)$, also
  \begin{equation}
    \begin{split}
      &\hspace*{-5mm}
      |\sigma_r'(y)-\frac12|\leq B\varepsilon\ ,\quad|\tau_r'(y)-\frac12|\leq
      B\varepsilon\quad\text{for }|y|\leq\varepsilon,
      \\
      B
      \varepsilon\geq
      \tau_{r}(y)  & \geq \sigma_{r}(y)\geq\left(  \frac{1}{2}-\kappa\varepsilon\right)
      (y+r)\geq0\text{ \quad for }y\in\lbrack-r,\varepsilon]\text{, and}\\
      0>\tau_{r}(y)  & \geq\sigma_{r}(y)\geq\left(  \frac{1}{2}+\varepsilon\right)
      (y+r)>-
      B
      \varepsilon\text{ \quad for }y\in[-\varepsilon,-r)\text{.}%
    \end{split}
    \label{Eq_estimate_tau_sigma_near_zero}%
  \end{equation}
  (Recall that $\sigma_{r}$ and $\tau_{r}$ share a zero at $z_r=-r$.)
  Finally, note that we can w.l.o.g. take $\overline\varepsilon$ so small that
  $\bar B:=\left(\frac12-\overline\varepsilon\right)\left(1+\frac
    B6-(\frac13+\frac{c}{6})\overline\varepsilon\right)\in(1,3]$.
  (Due to Assumption I we have $B\leq 25$.)\\

  Consider some $v=\int_Y w_{\bullet}\,d\nu$ with $\nu \in \mathsf{P}(Y)$.  We
  claim that for $\varepsilon\in(0,\frac{\overline\varepsilon}{\kappa})$
  \begin{equation}
    \label{eq:quadratic-estimate}
    |\phi(\widetilde{P}v)|
    \geq
    \bar B\cdot|\phi(v)|-\varepsilon^2\quad\text{if
    }\supp(\nu)\subseteq[-\varepsilon,\varepsilon].
  \end{equation}
  Denote $r:=G(\phi(v))$ which by S-shapedness of $G$ satisfies $|r| <
  B\varepsilon$.  In view of our system's symmetry, we may assume
  w.l.o.g. that $r\geq0$.  According to (\ref{Eq_mean_for_convex_combi}) and
  (\ref{Eq_P_tilda_acting_on_representing_measures}) we have
  \[
  \phi(\widetilde{P}v)
  =
  \phi\left(  \int_{{Y}} w_{\bullet}\,d (  \widetilde{\mathcal{L}}^{\ast}\nu )
  \right)  =\int_{{Y}}(\overline{w}\circ\sigma_{r})\cdot p_{r}\,d\nu+\int_{{Y}}(\overline
  {w}\circ\tau_{r})\cdot(1-p_{r})\,d\nu
  \]
  so that, due to (\ref{Eq_estimate_F_near_zero}) and
  (\ref{Eq_estimate_tau_sigma_near_zero}),
  \begin{align*}
    \int_{{Y}}&(\overline{w}\circ\sigma_{r})\cdot p_{r}\,d\nu\\
    & \geq
    \int_{{Y}}\left( \frac{1_{[-B\varepsilon,0)}\circ\sigma_{r}(y)}{6-6\varepsilon^2}
      +\frac{1_{[0,B\varepsilon]}\circ\sigma_{r}(y)}{6}\right)
    \cdot \sigma_r(y)\,p_{r}(y)\,d\nu(y)\\
    & \geq
    \int_{{Y}}\left( \frac{1_{[-\frac23,-r)}(y)}{6-6\varepsilon^2}\left( \frac{1}{2}
        +\varepsilon\right)+\frac{1_{[-r,\frac23]}(y)}{6}\left( \frac{1}{2}%
        -\kappa\varepsilon\right)\right)\cdot (y+r) \,p_{r}(y)\,d\nu(y)\text{.}%
  \end{align*}
  Combining this with the parallel estimate for $\int_Y(\overline{w}\circ\tau
  _{r})\cdot(1-p_{r})\,d\nu$, we get
  \begin{displaymath}
    \phi(\widetilde{P}v)
    \geq\int_{[-\frac{2}{3},-r)}\frac{\left(  \frac{1}{2}%
        +\varepsilon\right)  (y+r)}{6-6\varepsilon^2}\,d\nu(y)
    +\int_{[-r,\frac{2}{3}]}\frac{\left(
        \frac{1}{2}-\kappa\varepsilon\right)  (y+r)}{6}\,d\nu(y).
  \end{displaymath}
  Continuing, we find that
  \begin{align*}
    \phi(\widetilde{P}v)
    & \geq\int_{[-\frac{2}{3},-r)}\frac{\frac{1}{2}+\varepsilon}{6-6\varepsilon^2}\cdot
    y\,d\nu(y)+\int_{[-r,\frac{2}{3}]}\frac{\frac{1}{2}-\kappa\varepsilon}{6}\cdot
    y\,d\nu(y)+\frac{\frac{1}{2}-\kappa\varepsilon}{6}\,\cdot r\\
    & \geq\int_{[-\frac{2}{3},0)}\frac{\frac{1}{2}+\varepsilon}{6-6\varepsilon^2}\cdot
    y\,d\nu(y)+\int_{[0,\frac{2}{3}]}\frac{\frac{1}{2}-\kappa\varepsilon}{6}\cdot
    y\,d\nu(y)+\frac{\frac{1}{2}-\kappa\varepsilon}{6}\,\cdot r\\
    & \geq\int_{[-\frac{2}{3},0)}K\cdot\overline{w}(y)\,d\nu(y)
    +\int_{[0,\frac{2}{3}]}K^*\cdot\overline{w}(y)\,d\nu(y)
    +\frac{\frac{1}{2} -\kappa\varepsilon}{6}\,\cdot r\ ,
  \end{align*}
  where
  $K:=(\frac{1}{2}+\varepsilon)/(1-\varepsilon^2)>K^{\ast}:=(\frac{1}{2}-\kappa\varepsilon)(1-\varepsilon^2)$.
  As, because of (\ref{Eq_estimate_F_near_zero}),
  $\phi(v)=\int_Y\overline{w}\,d\nu\leq\frac\varepsilon5$, so that $r=G\left(
    \phi\left( v\right) \right) \geq(B-c\varepsilon) \cdot\phi\left(
    v\right)$, we conclude
  \begin{displaymath}
    \begin{split}
      \phi(\widetilde{P}v)
      &\geq
      \phi(v)\left(K^{\ast}+\frac{(\frac12-\kappa\varepsilon)(B-c\varepsilon)}6\right)
      +(K-K^{\ast})\int_{[-\frac{2}{3},0)}\overline{w}(y)\,d\nu(y)\\
      &\geq
      \phi(v)\left(\frac12-\overline\varepsilon\right)\left(1+\frac
        B6-\left(\frac13+\frac{c}{6}\right)\overline\varepsilon\right)
      -\varepsilon^2
      =
      \bar B\cdot\phi(v)-\varepsilon^2,
    \end{split}
  \end{displaymath}
  since $K-K^{\ast}\leq3\varepsilon$ and
  $|\overline{w}(y)|\leq\frac\varepsilon3$ whenever
  $|y|\leq\varepsilon\leq\frac13$.
  This proves \eqref{eq:quadratic-estimate}.\\

  Now take any $u\in W^s(u_0)\cap\Dc'$.  Then $\phi_n(u)\to0$, and the second
  alternative of Proposition~\ref{prop:summary} applies, so that
  $R_n(u)\leq\overline\varepsilon / \kappa$ and $(1+2BR_n(u))^2\leq\frac{\bar
    B+1}2$ for all $n$ larger than some $n_{\overline\varepsilon}$.  In
  particular,
  \begin{equation}
    \label{eq:Rnplus1}
    R_{n+1}(u)^2\leq(1+2BR_n(u))^2R_n(u)^2\leq\frac{\bar B+1}2R_n(u)^2
  \end{equation}
  for these $n$ in view of \eqref{Eq_estimate_tau_sigma_near_zero}.  Applying,
  for $n\geq n_{\overline\varepsilon}$, the estimate
  \eqref{eq:quadratic-estimate} to $v:=\widetilde{P}^n u$ and
  $\varepsilon:=R_n(u)$, we obtain
  \begin{displaymath}
    |\phi_{n+1}(u)|\geq\bar B\cdot|\phi_n(u)|-(R_n(u))^2 \quad
    \text{for } n\geq n_{\overline\varepsilon}.
  \end{displaymath}
  Suppose for a contradiction that $(R_n(u))^2<\frac{\bar B-1}2\,|\phi_n(u)|$
  for some $n>n_{\overline\varepsilon}$. Then $|\phi_{n+1}(u)|>\frac{\bar
    B+1}2|\phi_n(u)|$, and therefore $(R_{n+1}(u))^2\leq\frac{\bar
    B+1}2(R_n(u))^2<\frac{\bar B-1}2\frac{\bar B+1}2|\phi_n(u)|<\frac{\bar
    B-1}2|\phi_{n+1}(u)|$.  We can thus continue inductively to see that
  $|\phi_n(u)|<|\phi_{n+1}(u)|<|\phi_{n+2}(u)|<\dots$ which contradicts
  $\phi_n(u)\to0$. Therefore $|\phi_n(u)|\leq\frac2{\bar B-1}\,(R_n(u))^2$ for
  all $n>n_{\overline\varepsilon}$, and the assertion of our lemma follows.
\end{proof}

\begin{lemma}[{\textbf{$W^s(u_0)$ is a thin set for the order $\prec$}}]
  \label{lemma:stable-alternative}
  In the bistable regime, if $u=\int_Y w_{\bullet}\,d\mu$ and $v=\int_Y w_{\bullet}\,d\nu$ are densities in $\Dc'$ with $\mu\prec\nu$,
  then at most one of $u$ and $v$ can belong to $W^s(u_0)$.
\end{lemma}

\begin{proof}
  Suppose that $u\in W^s(u_0)$. We {are going to} show that
  $\widetilde{P}^n v\to u_{r_*}$,
  i.e. $\tilde\Lc^{*n}\nu\to\mu_{r_*}$ as $n\to\infty$.

  Assume for a contradiction that also $v\in W^s(u_0)$. We denote the parameters obtained
  from $u$ by $r_{n,\mu}:= {G}(\phi(\tilde
  P^{n-1}u)) ={G}(\int_Y\overline{w}\,d(\tilde\Lc^{*(n-1)}\mu))$, and define
  $r_{n,\nu}$ analogously. Then our assumption implies that
  $\lim_{n\to\infty}r_{n,\mu}=\lim_{n\to\infty}r_{n,\nu}=0$.

  In view of \eqref{eq:field_w_y},
  $\overline{w}'\geq\frac16$, and one checks immediately that
  $\inf_{Y}\sigma_0'=\frac{18}{49}>\frac13$ so that there is
  $n_0>0$ such that
  $\inf_{Y}\sigma_{r_{n,\mu}}'\geq\frac13$ for all
  $n\geq n_0$.
  Because of the strict monotonicity of $\tilde\Lc^*$ (Lemma \ref{lemma:pre-monotonicity}) we
  have $\tilde\Lc^{*n_0}\mu\prec\tilde\Lc^{*n_0}\nu$, so that
  (replacing $\mu$ and $\nu$ by these iterates) we can assume
  w.l.o.g. that $n_0=0$.  
  Denote
  $\Lc_\mu^{*(n)}:=\Lc^*_{r_{n,\mu}}\circ\dots\circ\Lc^*_{r_{1,\mu}}$ so that
  $\tilde\Lc^{*n}\mu=\Lc_\mu^{*(n)}\mu$ and ($\mu \mapsto r_{\mu}$ being non-decreasing)
  $\tilde\Lc^{*n}\nu\succeq\Lc_\mu^{*(n)}\nu$ for $n \geq 1$. Therefore

  \begin{displaymath}
   \begin{split}
    r_{n,\nu}-r_{n,\mu}
    &\geq
    G\left(\int_Y\overline{w}\,d(\Lc^{*(n)}_\mu\nu)\right)-G\left(\int_Y\overline{w}\,d(\Lc^{*(n)}_\mu\mu)\right)\\
    &\geq
    \inf_X G' \cdot \left( \int_Y\overline{w}\,d(\Lc^{*(n)}_\mu\nu)-\int_Y\overline{w}\,d(\Lc^{*(n)}_\mu\mu) \right).
   \end{split}
  \end{displaymath}
  In view of the lower bounds for $\overline{w}'$ and $\sigma_{r_{n,\mu}}',\tau_{r_{n,\mu}}'$,
  repeated application of the estimate \eqref{eq:gap-estimate} from Lemma \ref{lemma:gap-estimate} yields
  \begin{equation}
    \label{eq:r-estimate}
    r_{n,\nu}-r_{n,\mu}
    \geq
    \frac{\inf_X G'}{6\cdot3^n}\,\int_Y \id\,d(\nu-\mu).
  \end{equation}
  Observe that the last integral is strictly positive because $\mu\prec\nu$, cf. \eqref{eq:mono-fact-2a}.

  On the other hand, due to Proposition~\ref{prop:summary} there are $\varepsilon_n\searrow0$ such that
  \begin{displaymath}
    \supp(\tilde\Lc^{*n}\mu) \cup \supp(\tilde\Lc^{*n}\nu)\subseteq[-\varepsilon_n,\varepsilon_n] ,
  \end{displaymath}
  and as $\sigma_0'(0)=\frac12<\frac59$ and $r_{n,\mu},r_{n,\nu}\to0$
  (whence also $z_{r_{n,\mu}},z_{r_{n,\nu}} \to z_{0} =0$), there exists a
  constant $C>0$ such that
  $\varepsilon_n\leq C(\frac59)^n$ for $n\geq n'$. Hence $|\phi_n(u)|,|\phi_n(v)|\leq
  \max\{C_{u},C_{v}\}\cdot C^2(\frac{25}{81})^n$ for $n\geq n'$ by
  Lemma~\ref{lemma:roland}, and as
  $r_{n,\nu}-r_{n,\mu}\leq\sup\overline{w}'\cdot(|\phi_n(u)|+|\phi_n(v)|)$,
  this contradicts the previous estimate \eqref{eq:r-estimate}.
\end{proof}

We can now conclude this section with the

\begin{proof}[Proof of Proposition \ref{prop:basin-boundary}.]
  Suppose that $u=\int w_{\bullet}\,d\mu\in W^s(u_0)$. For $t\in(0,1)$ let $u^{(t)}:=\int_Y
  w_{\bullet}\,d((1-t)\mu+t\delta_{2/3}) \in \Dc'$. Then $u^{(t)}\succ u$, hence $u^{(t)} \not\in
  W^s(u_0)$ by the previous proposition. Therefore, due to Proposition~\ref{prop:summary}
  and monotonicity of $\tilde\Lc^{*}$, for any $t$, $\widetilde{P}^n u^{(t)}$ converges to
  $u_{r_*} \succ u_0$ as $n\to\infty$.

  On the other hand,
  $\lim_{t\to0}\|u-u^{(t)}\|_{L_1(X,\lambda)}=0$, so $u$ is in the boundary of
  the basin of $u_{r_*}$. Replacing $\delta_{2/3}$ by $\delta_{-2/3}$ yields
  the corresponding result for the basin of $u_{-r_*}$.
\end{proof}

\subsection{Differentiability of $\widetilde P$ at $\mathcal{C}^2$-densities}\label{subsec:differentiability}
As $\widetilde P$ is based on a parametrised family of PFOs where the
\emph{branches} of the underlying map (and not only their weights) depend on
the parameter, it is nowhere differentiable, neither as an operator on
$L_1(X,\lambda)$ nor as an operator on the space $\BV(X)$ of (much more
regular) functions of bounded variation on $X$.  On the other hand, as the
branches of the map and their parametric dependence are analytic, one can show
that $\widetilde P$ is differentiable as an operator on the space of functions
that can be extended holomorphically to some complex neighbourhood of
$X\subseteq\C$.

Here we will focus on a more general but slightly weaker differentiability
statement.

\begin{lemma}[\textbf{Differentiability of $\widetilde P$ at $C^2$-densities}]
  \label{lemma:differentiability}
  Let $u\in \mathcal{C}^2(X)$ be a probability density w.r.t. $\lambda$ and
  let $g\in L_1(X,\lambda)$ have $\int_X g\,d\lambda=0$. Then
  \begin{equation}
    \label{eq:gateaux}
    \frac{\partial}{\partial\tau}\widetilde P(u+\tau g)|_{\tau=0}
    =
    P_r(g)+w_r(u)\cdot G'(\phi(u))\,\phi(g)
  \end{equation}
  where $r=G(\phi(u))$, $w_{r}(u):=P_r\left((u\,v_r)'\right)$, and
  $v_r(x)=\frac{4x^2-1}{4-r^2}$. If we consider $\widetilde P$ as an operator
  from $\BV(X)$ to $L_1(X,\lambda)$, then $\widetilde P$ is even
  differentiable at each probability density $u\in \mathcal{C}^2(X)\subset\BV(X)$ and
  \begin{equation}
    \label{eq:diff-BV-L1}
    D\widetilde P|_u=P_r+G'(\phi(u))\,w_r(u)\otimes\phi\ .
  \end{equation}
\end{lemma}

\begin{proof}
In order to simplify the notation define
  a kind of transfer operator $L$ by $L u:=u+u\circ
  f_{{1\,1\choose0\,1}}$ and note that $(Lu)'=Lu'$. Observing that $f_{N_r^{-1}}=f_{M_r^{-1}}\circ
  f_{{1\,1\choose0\,1}}$, we have $P_ru=L(u\circ f_{M_r^{-1}}\cdot
  f_{M_r^{-1}}')$. Define
  \begin{equation}
    \label{eq:v_r}
    v_{r}(x):=\left(\frac{\partial}{\partial r}f_{M_r^{-1}}\right)(f_{M_r}(x))
    =
    \frac{4x^2-1}{4-r^2}.
  \end{equation}
{}For a function
  $u\in \mathcal{C}^2(X)$ denote by $U$ the antiderivative of $u$. Then
  \begin{equation}
    \label{eq:u-diff}
    \begin{split}
      u\circ f_{M_s^{-1}}&\cdot f_{M_s^{-1}}'
      -u\circ f_{M_r^{-1}}\cdot f_{M_r^{-1}}'
      =
      \left(U\circ f_{M_s^{-1}}-U\circ f_{M_r^{-1}}\right)'\\
      =&
      \left((s-r)\cdot\frac{\partial}{\partial r}(U\circ f_{M_r^{-1}})
        +R_{s,r}\right)'
    \end{split}
  \end{equation}
  where
  \begin{displaymath}
    R_{s,r}(x):=
    \int_r^s(s-t)\,\frac{\partial^2}{\partial
      t^2}(U(f_{M_t^{-1}}(x)))\,dt\ .
  \end{displaymath}
  As $\frac{\partial}{\partial r}(U\circ f_{M_r^{-1}})=u\circ
  f_{M_r^{-1}}\cdot\frac{\partial}{\partial r}f_{M_r^{-1}}
  =\left(u\,v_r\right)\circ f_{M_r^{-1}}$, we have
  \begin{displaymath}
    \left(\frac{\partial}{\partial r}(U\circ f_{M_r^{-1}})\right)'
    =
    \left(u\,v_r\right)'\circ
    f_{M_r^{-1}}\cdot f_{M_r^{-1}}'\ .
  \end{displaymath}
  Together with \eqref{eq:u-diff} this yields
  \begin{displaymath}
    \begin{split}
      P_su-P_ru
      &=
      L\left(u\circ f_{M_s^{-1}}\cdot f_{M_s^{-1}}'
        -u\circ f_{M_r^{-1}}\cdot f_{M_r^{-1}}'\right)\\
      &=
      (s-r)\,L\left((u\,v_r)'\circ f_{M_r^{-1}}\cdot f_{M_r^{-1}}'\right)
      +L\left(R_{s,r}'\right)\\
      &=
      (s-r)\,P_r\left((u\,v_r)'\right)+L\left(R_{s,r}'\right)
    \end{split}
  \end{displaymath}
  and $|L(R_{s,r}')(x)|\leq C\,(s-r)^2$ with a constant that involves only the
  first two derivatives of $u$.

  Now let $u\in \mathcal{C}^2(X)$ be a probability density, and let
  $g\in L_1(X,\lambda)$
  be such that $\int g\,d\lambda=0$. Let $r:=G(\phi(u))$ and
  $s:=G(\phi(u+g))$. Then
  \begin{equation}
    \begin{split}
      \widetilde P(u+g)-\widetilde P(u)
      &=
      (P_s u-P_r u)+P_rg+(P_s g-P_r g)\\
      =
      (s-r)&\,P_r\left((u\,v_r)'\right)+P_r(g)+(P_s g-P_r g)+L(R_{s,r}')\ .
    \end{split}
  \end{equation}
  This implies at once formula \eqref{eq:gateaux} for the directional
  derivative, and as $\|P_sg-P_rg\|_1\to0$ ($s\to r$) uniformly for $g$ in the
  unit ball of $\BV(X)$, also \eqref{eq:diff-BV-L1} follows at once.
\end{proof}
\begin{prop}[ \textbf{$u\equiv1$ is a hyperbolic fixed point of $\widetilde
    P$}]
  \label{prop:hyperbolic-fixed}
  In the bistable regime, $u\equiv1$ is a hyperbolic fixed point of
  $\widetilde P|_{\Dc\cap\BV(X)}$ in the following sense: the derivative of
  $\widetilde P:\Dc\cap\BV(X)\to L_1(X,\lambda)$ at $u\equiv1$ has a
  one-dimensional unstable subspace and a codimension $1$ stable subspace.
\end{prop}
\begin{proof}
Let $Q:=D\widetilde P|_{u\equiv1}$. As $G'(0)=B$ and $w_0(1)=P_0[2x]=[x]$,
it follows from \eqref{eq:diff-BV-L1} that $Q=P_0+B\,[x]\otimes\phi$. (Here
$[2x]$ denotes the function $x\mapsto2x$, etc.) Observe now that
$\phi([x])=\frac1{12}$. Then $Q[x]=P_0[x]+\frac B{12}\,[x]=(\frac12+\frac
B{12})[x]$ so that, for $B>6$, $Q$ has the unstable eigendirection $[x]$ with
eigenvalue $\lambda:=\frac12+\frac{B}{12}>1$. On the other hand, as
$\phi(1)=0$, we have $Q1=P_01=1$, so the constant density $1$ is a neutral
eigendirection, and finally, for $f\in\ker(\phi)\cap\ker(\lambda)$, we have
$Qf=P_0f$, so $\Var(Qf)\leq\frac12\Var(f)$.
\end{proof}

\section{The noisy system}
\label{sec:noisy}
In Theorem~\ref{theo:passage-to-infinity} we proved that, in the bistable
regime, each weak accumulation point of the sequence
$($\boldmath$\mu$\unboldmath$_N\circ\epsilon_N^{-1})_{N\geq1}$ is of the form
$\alpha\,\delta_{u_{-r_*}\lambda}+(1-2\alpha)\,\delta_{u_0\lambda}+\alpha\,\delta_{u_{r_*}\lambda}$
for some $\alpha\in[0,\frac12]$, i.e. that the stationary states of the
finite-size systems approach a mixture of the stationary states of the
infinite-size system.  It is natural to expect that actually $\alpha=\frac12$,
meaning that any limit state thus obtained is a mixture of \emph{stable}
stationary states of $\widetilde P$.  While we could not prove this for the
model discussed so far, we now argue that this conjecture can be verified if
we add some noise to the systems.

At each step of the dynamics we
perturb the parameter of the single-site maps by a small amount. To make this
idea more precise, let
\begin{equation}
  \begin{split}
      r(Q,t)&=G(\phi(Q)+t)\quad\text{for }Q\in{\mathsf P}(X)\text{ and
      }t\in\R\,,\text{ in particular}\\
      r(\x,t)&=G(\phi(\x)+t)\quad\text{for }\x\in X^N\text{ and }t\in\R\,.
  \end{split}
\end{equation}
Let $\eta_1,\eta_2,\dots$ be i.i.d. symmetric real valued random variables
with common distribution $\varrho$ and $|\eta_n|\leq\varepsilon$. For
$n=1,2,\dots$ and $\x\in X^N$ let us define the $X^N$-valued Markov process
$(\xi_n)_{n\in\N}$ by $\xi_0=\x$ and
\begin{equation}
  \xi_{n+1}=T_{r(\xi_n,\eta_{n+1})}(\xi_n).
\end{equation}
Assume now that the distribution of $\xi_n$ has density $h_n$ w.r.t. Lebesgue
measure on $X^N$. Then routine calculations show that the distribution of
$\xi_{n+1}$ has density $\int_\R P_{N,t}h_n\,d{\varrho}(t)$ where $P_{N,t}$ is
the PFO of the map $\Tb_{N,t}:X^N\to X^N$,
$(\Tb_{N,t}(\x))_i=T_{r(\x,t)}(x_i)$. It is straightforward to check that, for
sufficiently small $\varepsilon$, Lemmas~\ref{lemma:homeo} --
\ref{lemma:inv-contraction} from Section~\ref{sec:finite-sys-proofs} carry
over to all $\Tb_{N,t}$ ($|t|\leq\varepsilon$) with uniform bounds, and that
$\int_{X^N}|P_{N,t}f-P_{N,0}f|d\lambda^N\leq\text{const}_N\cdot\varepsilon\cdot\Var(f)$
so that the perturbation theorem of \cite{Keller82} guarantees that the
process $(\xi_n)_{n\in\N}$ has a unique stationary probability
\boldmath$\mu$\unboldmath$_{N,\varepsilon}$ whose density w.r.t. $\lambda^N$
tends, in $L_1(X^N,\lambda^N)$, to the unique invariant density of $\Tb_N$ as
$\varepsilon\to0$.  This convergence is not uniform in $N$,
however. Nevertheless, folklore arguments show that there is some
$\tilde\varepsilon>0$ such that, for all $\varepsilon\in(0,\tilde\varepsilon)$
and all $N\in\N$ the absolutely continuous stationary measure
\boldmath$\mu$\unboldmath$_{N,\varepsilon}$ is unique so that the symmetry
properties of the maps $T_r$ and the random variables $\eta_n$ guarantee that
\boldmath$\mu$\unboldmath$_{N,\varepsilon}$ is symmetric in the sense that its
density $h_{N,\varepsilon}$ satisfies
$h_{N,\varepsilon}(x)=h_{N,\varepsilon}(-x)$.

On the other hand, for each fixed $\varepsilon>0$, all weak limit points of the measures
\boldmath$\mu$\unboldmath$_{N,\varepsilon}\circ\epsilon_N^{-1}$ as
$N\to\infty$ are stationary probabilities for the $\mathsf{P}(X)$-valued Markov process
$(\Xi_n)_{n\in\N}$ defined by
\begin{equation}
  \Xi_{n+1}=\Xi_n\circ T_{r(\Xi_n,\eta_{n+1})}^{-1}\ ,
\end{equation}
compare the definition of $\widetilde T : {\mathsf{P}}(X) \to {\mathsf{P}}(X)$
in \eqref{eq:r_mu}. The proof is completely analogous to the corresponding one
for the unperturbed case (see Lemma~\ref{le:contTtil} and
Corollary~\ref{coro:contTtil}).  For $\varepsilon\in(0,\tilde\varepsilon)$ the
symmetry of the \boldmath$\mu$\unboldmath$_{N,\varepsilon}$ carries over to
these limit measures $Q$ in the sense that $Q(A)=Q\{\hat\mu:\mu\in A\}$ for
each Borel measurable set $A\subseteq\mathsf{P}(X)$ where
$\hat\mu(U):=\mu(-U)$ for all Borel subsets $U\subseteq X$.

The following proposition then shows that, in the bistable regime and for
small $\varepsilon>0$ and large $N$, the measures
\boldmath$\mu$\unboldmath$_{N,\varepsilon}$ are weakly close to the mixture
$\frac12\left((u_{-r_*}\lambda)^\N+(u_{r_*}\lambda)^\N\right)$ of the stable
states for $\widetilde P$; compare also
Theorem~\ref{theo:passage-to-infinity}.

\begin{prop}[\textbf{Invariant measures for infinite-size noisy systems}]
  Suppose $G'(0)>6$ so that we are in the bistable regime and recall that the
  $\eta_n$ are symmetric random variables.

  Then, for every $\delta>0$
  there is $\varepsilon_0>0$ such that for each
  $\varepsilon\in(0,\varepsilon_0)$ the stationary distribution
  $Q_\varepsilon$ of $\Xi_{n}$
  on ${\mathsf P}(X)$ is supported on the set of measures $u\cdot\lambda\in{\mathsf
    P}(X)$ which have density $u=\int_Yw_\bullet\,d\mu\in\Dc'$ with
  representing measures $\mu\in{\mathsf P}(Y)$ satisfying
  $\d(\mu,\frac12(\mu_{-r_*}+\mu_{r_*}))\leq\delta$.
\end{prop}
\begin{proof}[Sketch of the proof]
  Let $Q$ be a stationary distribution of $\Xi_n$ that occurs as a weak limit
  of the measures \boldmath$\mu$\unboldmath$_{N,\varepsilon}$. So $Q$ is
  symmetric.  Just as in the proof of Theorem~\ref{theo:passage-to-infinity},
  where the ``zero noise limit'', namely the transformation $\widetilde T$ is
  treated, one argues that $Q$ is supported by the set of measures
  $u\cdot\lambda$, $u\in\Dc$. Arguing as in the derivation of
  \eqref{eq:shadow-estimate} one shows that densities in the support of $Q$
  can be approximated in $L_1(X,\lambda)$ by densities from $\Dc'$, and the
  stationarity of $Q$ implies that $Q$ is indeed supported by measures with
  densities from $\Dc'$. Therefore the process $(\Xi_n)_{n\geq0}$ can be
  described by the transfer operator $\widetilde\Lc_\varepsilon^*$ of an
  iterated function system on $Y$ just as the self-consistent PFO $\widetilde
  P$ is described by the operator $\widetilde\Lc^*$ in equation
  \eqref{Eq_P_tilda_acting_on_representing_measures}. The only difference is
  that in this case one first chooses the parameter $r$ randomly,
  $r=G(\int_Y\overline{w}\,d\mu+\eta_{n+1})$ and then the branch $\sigma_r$ or
  $\tau_r$ with respective probabilities $p_r(y)$ and $(1-p_r(y))$.

  Let $y>0$ be such that $\mathbb{P}\{\eta_n>y\}>0$. Suppose now that for some
  realisation of the process $(\Xi_n)_{n\geq0}$ the numbers
  $r(\Xi_n,\eta_{n+1})$ satisfy condition (${\clubsuit}_\varepsilon$) of
  Lemma~\ref{lemma:crucial} for all $\varepsilon>0$. Then it follows, as in
  the proof of Proposition~\ref{prop:summary}, that
  $\lim_{n\to\infty}r_n(\Xi_n,\eta_{n+1})=0$ and the measures $\Xi_n$ converge
  weakly to $\lambda$ so that also $\lim_{n\to\infty}r(\Xi_n,0)=0$. As
  $\eta_n>y>0$ for infinitely many $n$ almost surely, both limit cannot be
  zero at the same time, and we conclude that almost surely there is some
  $\varepsilon>0$ such that (${\clubsuit}_{\varepsilon}$) is not satisfied. In
  particular, there are $\bar\varepsilon>0$ and $\bar n\in\N$ such that
  (${\clubsuit}_{\bar\varepsilon}$) is violated for $n=\bar n-1$ with some
  positive probability $\kappa$.

  Let $\Xi_n=h_n\cdot\lambda$ with $h_n=\int_Yw_\bullet\,d\nu_n$. (So $h_n$
  and $\nu_n$ are random objects.)  As in \eqref{eq:a-lower} we conclude that
  $\sup\supp(\nu_{\bar n})<-\bar\varepsilon/3$ or $\inf\supp(\nu_{\bar
    n})>\bar\varepsilon/3$ in this case. Without loss of generality we assume
  that the latter happens with probability at least $\frac{\kappa}{2}$.

  Next, as in \eqref{eq:jump-from-y} we may choose $y\in(0,\bar\varepsilon/3)$
  so small that
  $0<y<y_1:=\sigma_{G(\overline{w}(y))}(y)\leq\inf\supp(\tilde\Lc^*\delta_y)$. Hence,
  for reasons of continuity, there is $\varepsilon_1>0$ such that also
  $y\leq\inf\supp(\tilde\Lc^*_\varepsilon\delta_y)$ if
  $\varepsilon\in[0,\varepsilon_1)$. Therefore, in view of the monotonicity of
  the operator $\tilde\Lc^*_\varepsilon$, we can conclude that
  $\inf\supp(\nu_n)\geq y$ for all $n\geq\bar n$ with probability at least
  $\frac{\kappa}{2}$.  Now fix $\delta>0$.  By
  Lemma~\ref{lemma:controlled-convergence}\ref{item:controlled-convergence-b}
  there is some (non-random) $n_1\in\N$ such that
  $\d(\tilde\Lc^{*n_1}\nu_{n},\mu_{r_*})\leq\frac\delta2$ for all $n\geq\bar
  n$ with probability at least $\frac{\kappa}{2}$. But then, by continuity
  reasons again, there is $\varepsilon_0\in(0,\varepsilon_1)$ such that
  $\d(\nu_{n+n_1},\mu_{r_*})<\delta$ for all $n\geq\bar n$ with probability at
  least $\frac{\kappa}{2}$. The claim of the proposition follows now, because
  $(\Xi_n)_n$ is a Markov process and because the stationary distribution $Q$
  is symmetric.
\end{proof}

\appendix

\section{{Some technical and numerical results}}

\subsection{Proof of Lemma \ref{le:contTtil}}
\label{se:conTtil}

 It suffices to prove the convergence for evaluations of
any Lipschitz continuous function $\varphi$ defined on $X$. Let us
denote $r_n=r(Q_n)$ (resp. $r=r(Q)$), and $\alpha_n$ (resp. $\alpha$) the
discontinuity point of $T_{r_n}$ (resp. $T_r$). Recall that
$\alpha_n=-\frac{r_n}4$ (resp. $\alpha=-\frac{r}4$).

Let us fix $\varepsilon>0$. $Q$ being non-atomic, there exists
$\delta>0$ such that the interval $U:=[c-\delta,c+\delta]$ is of
$Q$-measure smaller that $\varepsilon$. The weak convergence of
$Q_n$ to $Q$ implies that $c_n$ tends to $c$, and that
$\limsup_{n\rightarrow+\infty}Q_n(U)\leq Q(U)$. Let us choose
$n_0$ such that for all $n\geq n_0$, $|c_n-c|<\frac{\delta}2$ and
$Q_n(U)<\varepsilon$. One then has
\begin{equation}
  \begin{split} \Big|\int_X \varphi \,d({\widetilde
T}Q)-\int_X \varphi \,d({\widetilde T}Q_n)\Big|=\Big|\int_X
\varphi\circ T_r \,dQ-\int_X \varphi\circ T_{r_n} \,d Q_n\Big|\\
\leq \Big|\int_X \varphi\circ T_r \,d(Q-Q_n)\Big|+\Big|\int_{U^c}
(\varphi\circ T_{r_n}-\varphi\circ
T_{r}) \,d Q_n\Big|\\+\Big|\int_{U} (\varphi\circ
T_{r_n}-\varphi\circ T_{r}) \,d Q_n\Big|\\ \leq \Big|\int_X
\varphi\circ T_r \,d(Q-Q_n)\Big|+\operatorname{Lip}(\varphi)\sup_{U^c}|T_{r_n}-T_r|
+2\varepsilon\|\varphi\|_\infty
  \end{split}
\end{equation} Since the application $\varphi\circ T_r$ as a single
discontinuity point, which is of zero $Q$-measure, the first term
converges to zero. The second one also goes to zero since it measures
the dependence of $T_r$ on its parameter away from the discontinuity
point (one can make an explicit computation).

\subsection{The fields of the densities $u_r$}
\label{ssec:infinite-sys-proofs-1}
We start with some observations on the function $\psi(r):=\phi(u_r)$ that are
based on symbolic computations and on numerical evaluations. One finds
\begin{equation}
  \label{eq:r-to-phi-hr}
  \psi(r)
  =
  \frac1r+\frac{\log\left( \frac{4+4\,r-3\,{r}^{2}}{4-4\,r-3\,{r}^{2}}\right) }
  {\log\left( \frac{4-9\,{r}^{2}}{4-{r}^{2}}\right) }
  =
  \frac{r}{6}+\frac{7\,{r}^{3}}{40}+\frac{461\,{r}^{5}}{2016}
  +\frac{4619\,{r}^{7}}{13440}+\dots\ .
\end{equation}
{}From this numerical evidence (see Figure~\ref{fig:plots-1} for a plot) it is
clear that, for $r\in[0,0.4]$,
\begin{gather*}
  \psi(r)\geq\frac r6,\quad\text{and}\\
  \frac{\psi''(r)}{(\psi'(r))^2}
  =\frac{189\,r}{5}-\frac{12862\,{r}^{3}}{175}
  +\frac{44487\,{r}^{5}}{500}-\frac{346403009\,{r}^{7}}{4042500}+\dots
  \leq\frac{189}5r.
\end{gather*}
Hence $H'(r)=G'(\psi(r))\,\psi'(r)\leq G'(\frac r6)\,\psi'(r)$.
 As
\begin{displaymath}
  H''=(G\circ\psi)''=\left(\frac{G''}{G'}\circ\psi+\frac{\psi''}{(\psi')^2}\right)\cdot(\psi')^2\cdot(G'\circ\psi)\ ,
\end{displaymath}
$H''(r)\leq0$ follows provided
$\frac{G''(\psi(r))}{G'(\psi(r))}\leq-\frac{189}5r$.
Therefore, assumption \eqref{eq:S-shape} is fulfilled, if
\begin{equation}
  G'(x)\leq\frac1{\psi'(6x)}\quad\text{or if}\quad
  \frac{G''(x)}{G'(x)}\leq-\frac{189}5\cdot6x
\end{equation}

{}For
$G(x)=A\tanh(\frac BAx)$, in which case $G'(x)=B/\cosh(\frac BA x)^2$ and
$\frac{G''(x)}{G'(x)}=-2\frac{B}{A}\tanh(\frac BAx)$
this can be checked numerically. (Observe that $0\leq A\leq 0.4$ and
distinguish the cases $B=G'(0)\leq6$ and $B>6$.)
{}For an illustration see the rightmost plot of $H(r)$ in
{}Figure~\ref{fig:plots-1}.

\begin{figure}
  \begin{minipage}{0.325\linewidth}
  \centering
    \includegraphics[width=0.95\linewidth,height=0.95\linewidth]{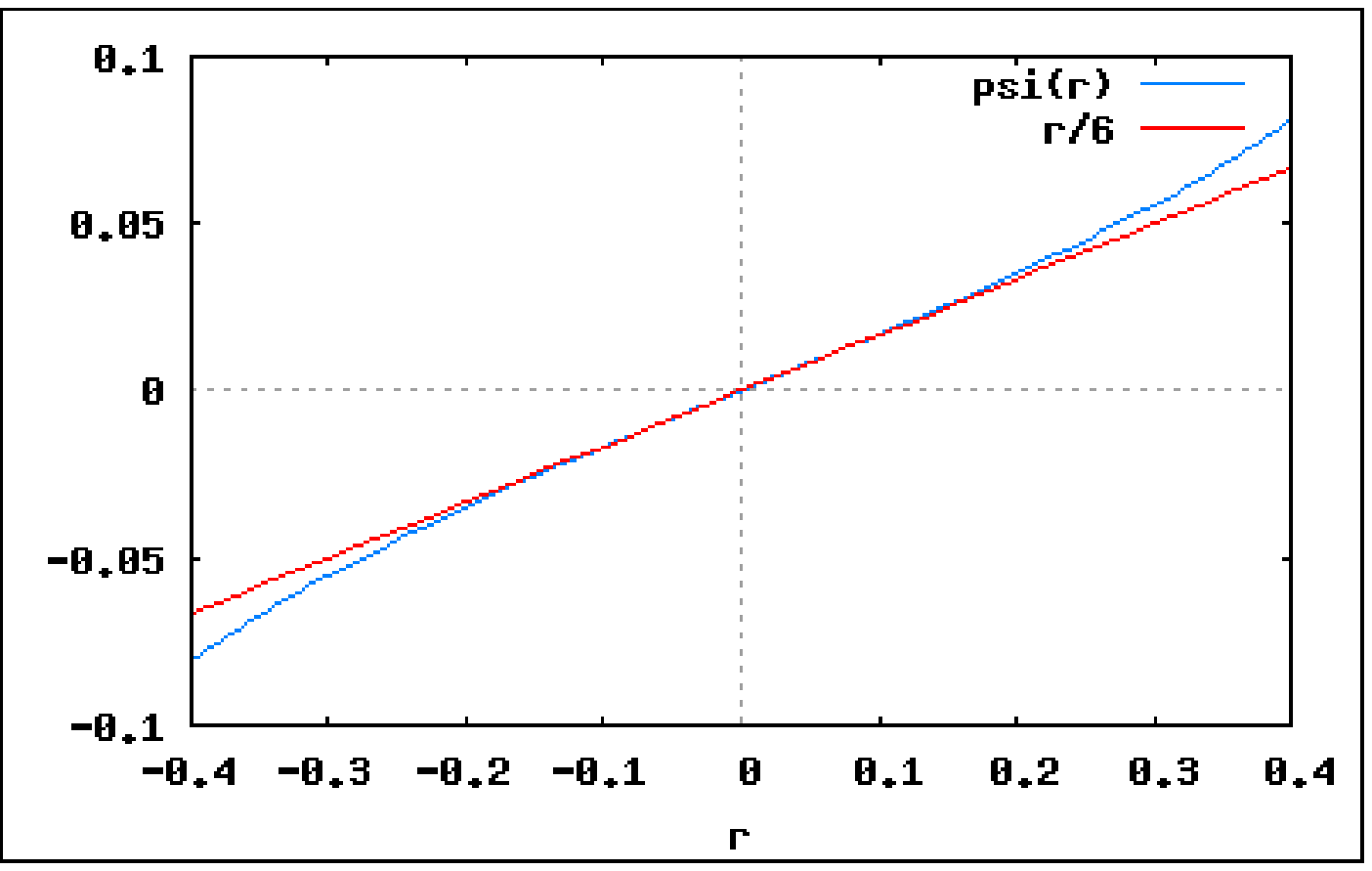}
  \end{minipage}
  \begin{minipage}{0.325\linewidth}
  \centering
    \includegraphics[width=0.95\linewidth,height=0.95\linewidth]{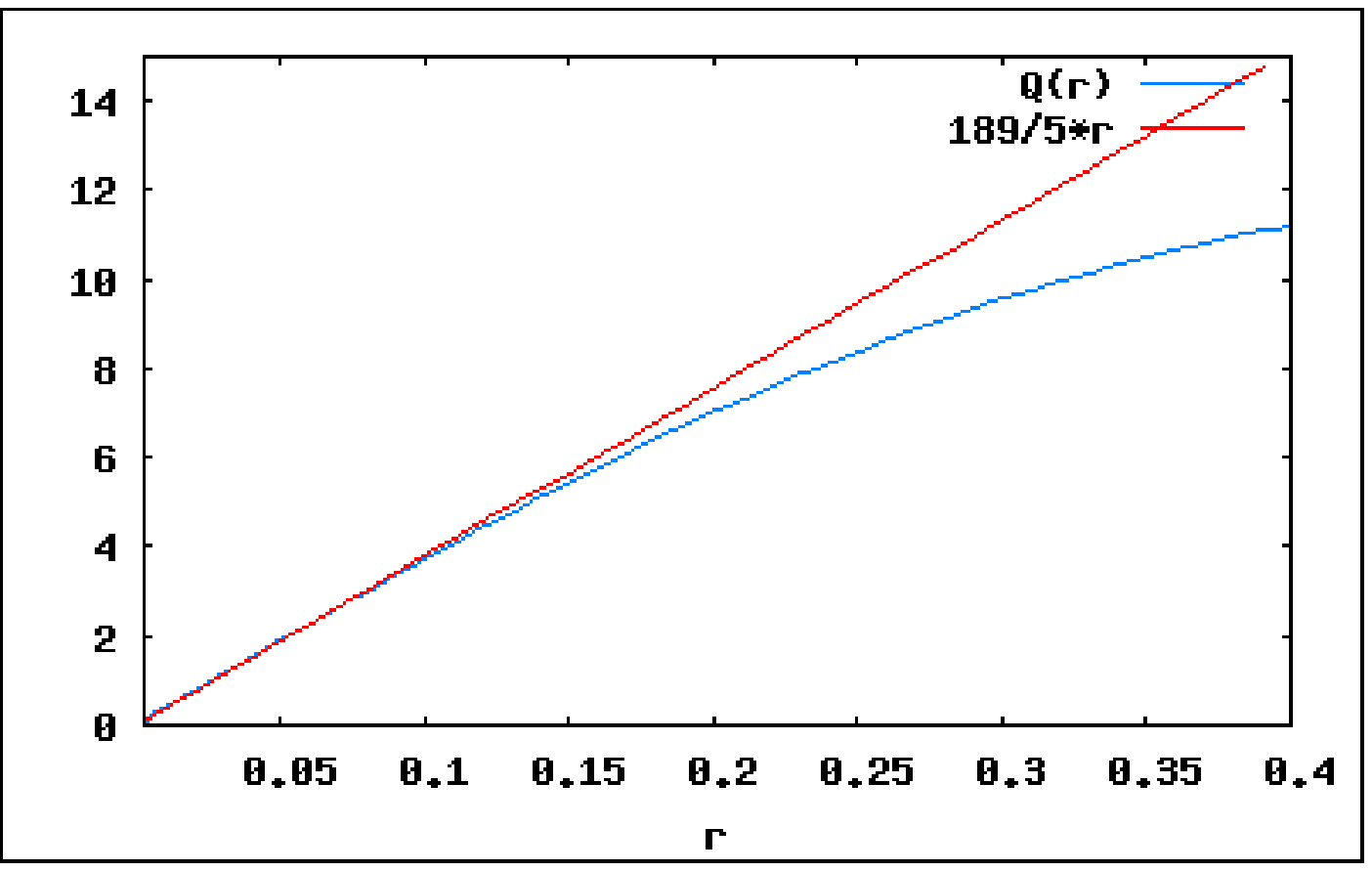}
  \end{minipage}
  \begin{minipage}{0.325\linewidth}
  \centering
    \includegraphics[width=0.95\linewidth,height=0.95\linewidth]{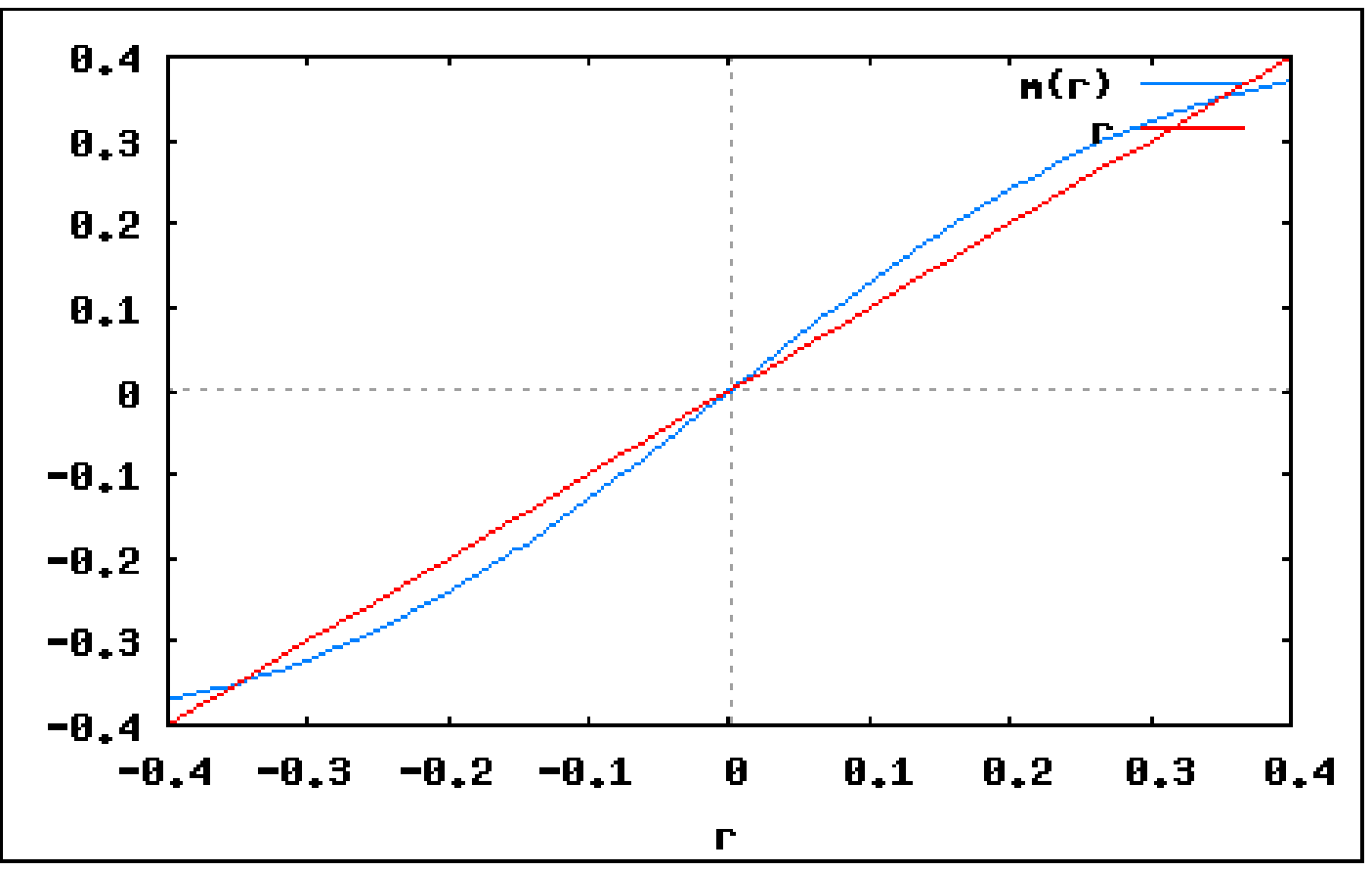}
  \end{minipage}
  \caption{The functions $\psi(r):=\phi(u_r)$ (left),
    $Q(r):=\frac{\psi''(r)}{(\psi'(r))^2}$ (centre), and $H(r)=A\tanh(\frac
    BA\phi(u_r))$ with $A=0.4$ and $B=8$ (right).}
  \label{fig:plots-1}
\end{figure}

\end{document}